\documentclass[12pt,leqno,twoside]{article}
\usepackage{amssymb}
\usepackage{amsmath}
\usepackage{amsthm,amscd}
\usepackage{bbm}
\usepackage{t1enc}
\usepackage{a4,indentfirst,latexsym}
\usepackage{graphics}
\usepackage[mathscr]{eucal}
\usepackage{mathrsfs}
\usepackage{cite,enumitem,graphicx}
\usepackage[colorlinks=true,urlcolor=blue,
citecolor=blue,linkcolor=blue,linktocpage,pdfpagelabels,
bookmarksnumbered,bookmarksopen]{hyperref}
\usepackage[colorinlistoftodos]{todonotes}
\usepackage{color}

\allowdisplaybreaks %allowbreaks of equations

\newtheorem{theorem}{Theorem}[section]
\newtheorem{corollary}[theorem]{Corollary}

\newtheorem{lemma}[theorem]{Lemma}
\newtheorem{proposition}[theorem]{Proposition}

\theoremstyle{definition}
\newtheorem{definition}[theorem]{Definition}
\newtheorem{remark}[theorem]{Remark}

\newcommand{\oH}{
	\mathring{\mathrm{H}}^1}

\newcommand\supp{\mathrm{supp}}
\newcommand\la{\langle}
\newcommand\ra{\rangle}

\def\R{\mathbb{R}}
\def\SS{\mathbb{S}}
\def\B{\mathbb{B}}

\def\i{\mathrm{i}}

\newcommand{\bW}{{\mathbf W}}
\newcommand{\bphi}{\boldsymbol{\phi}}
\newcommand{\bh}{{\mathbf h}}
\newcommand{\bu}{{\mathbf u}}

\newcommand{\rL}{{\mathrm L}}
\newcommand{\rH}{{\mathrm H}}

\newcommand{\cA}{{\mathcal A}}

\newcommand{\cH}{{\mathcal H}}
\newcommand{\cI}{{\mathcal I}}

\newcommand{\cL}{{\mathcal L}}

\newcommand{\cN}{{\mathcal N}}
\newcommand{\cO}{{\mathcal O}}

\newcommand{\cR}{{\mathcal R}}
\newcommand{\cS}{{\mathcal S}}
\newcommand{\cT}{{\mathcal T}}

\newcommand{\fL}{{\mathfrak L}}

\newcommand{\fR}{{\mathfrak R}}

\newcommand{\Si}{\Sigma}

\newcommand{\intsigma}{{\mathring\Si}}

\newcommand{\sumi}{\substack{i'=1 \\i' \neq i}} 

\newcommand{\sumj}{\substack{j'=1 \\j' \neq j}}

\title
{Blow-up Solutions for General Toda Systems on Riemann Surfaces}  

\author{ Zhengni Hu 
	and Miaomiao Zhu}

\begin{document}
	
	\maketitle
	%%%%%%%%%%%%%%%%%%%%%%%%%%%%%%%%%%%%%%%%%%%%%%%%%%%%%%%
	%             5. ABSTRACT
	%%%%%%%%%%%%%%%%%%%%%%%%%%%%%%%%%%%%%%%%%%%%%%%%%%%%%%%
	\begin{abstract}
		In this paper, we study  general Toda systems with homogeneous Neumann boundary conditions on Riemann surfaces. Assuming the surface satisfies the  ``$k$-symmetric'' condition, we construct a family of bubbling solutions using singular perturbation methods, where the concentration rates of different components occur in distinct orders.
		In particular, we establish the existence of asymmetric blow-up solutions for the $SU(3)$ Toda system. Furthermore, the blow-up points are precisely located at the ``$k$-symmetric'' centers of the surface.\\
		\medskip
		
		% Please provide a minimum of 5 keywords or phrases.
		
		% 2020 MSC numbers are required.
	\end{abstract}

	%%%%%%%%%%%%%%%%%%%%%%%%%%%%%%%%%%%%%%%%%%%%%%%%%%%%%%
	%                   6. BODY
	%%%%%%%%%%%%%%%%%%%%%%%%%%%%%%%%%%%%%%%%%%%%%%%%%%%%%%
	
	% Only the first word and proper nouns of section titles should be capitalized.
	% The title of section 1:
	\section{Introduction}\label{sec:1}
	Given a Riemann surface $
	(\Sigma,g)$ with smooth boundary 
	$\partial \Sigma$ (possible to be empty) and a positive integer $N\geq 2$, we consider the following  Toda system equipped with homogeneous Neumann boundary conditions
	\begin{equation}\label{eq:toda_mfN}
		\begin{cases}
			-\Delta_g u_1 = \sum_{j=1}^N a_{ij}\rho_j\left( \frac{V_je^{u_j}}{\int_{\Sigma} V_je^{u_j} dv_g}-\frac 1 {|\Sigma|_g} \right)    & \text{ in } \intsigma\\
			\partial_{\nu_g} u_1=\dots= \partial_{\nu_g} u_N= 0 & \text{ on } \partial\Sigma
		\end{cases},
	\end{equation}
	where 
	$\intsigma=\Sigma\setminus\partial\Sigma$ is the interior of $\Sigma$, $\nu_g$  is the outward unit normal vector on $\partial\Sigma$, $\Delta_g$ is the Laplace-Beltrami operator, $dv_g$  is the Riemannian volume element, $|\Sigma|_g=\int_{\Sigma} dv_g$ denotes the total volume of $\Sigma$, 	$\rho_i$ is a non-negative parameter,  $V_i:\Sigma\rightarrow \R$ is smooth positive functions, and the matrix 
	$(a_{ij})$ is one of the Cartan matrices of a general simple Lie algebra with rank $N$. 	For simplicity, we normalize the volume of $\Sigma$, i.e., $|\Sigma|_g =1$.
	
	In this paper, we consider the case where the Cartan matrix 
	$(a_{ij})$ takes one of the following forms:
	\[
	\mathbf{A}_N =
	\begin{pmatrix}
		2 & -1 & 0 & \dots & 0 \\
		-1 & 2 & -1 & \dots & 0 \\
		\vdots & \vdots & \ddots & \ddots & \vdots \\
		0 & \dots & -1 & 2 & -1 \\
		0 & \dots & 0 & -1 & 2
	\end{pmatrix},
	\quad
	\mathbf{B}_N =
	\begin{pmatrix}
		2 & -1 & 0 & \dots & 0 \\
		-1 & 2 & -1 & \dots & 0 \\
		\vdots & \vdots & \ddots & \ddots & \vdots \\
		0 & \dots & -1 & 2 & -2 \\
		0 & \dots & 0 & -1 & 2
	\end{pmatrix},
	\]
	\[
	\mathbf{C}_N =\left(
	\begin{matrix}
		2 & -1 & 0 & \dots & 0 \\
		-1 & 2 & -1 & \dots & 0 \\
		\vdots & \vdots & \ddots & \ddots & \vdots \\
		0 & \dots & -1 & 2 & -1 \\
		0 & \dots & 0 & -2 & 2
	\end{matrix}\right), \text{ and }
	\mathbf{G}_2:= \begin{pmatrix}
		2&-1\\
		-3&2
	\end{pmatrix}.\]
	Among these, the system \eqref{eq:toda_mfN} associated with the Cartan matrix $\mathbf A_N$ is usually referred to as the $SU(N+1)$ Toda system. 
	Moreover, the arguments developed in this paper can be extended, with only minor modifications, to Toda systems associated with more general Cartan matrices. 
	For simplicity of presentation, we restrict ourselves to the Cartan matrices considered above.

	Toda system \eqref{eq:toda_mfN} has been extensively studied over the past few decades due to its deep connections with various areas in geometry and physics. In differential geometry, it is related to the theory of harmonic maps and holomorphic curves into complex projective spaces $\mathbb{C}\mathbb{P}^N$ (see~\cite{guest1997harmonic,bolton1988conformal,doliwa1997holomorphic,chern1987harmonic,bolton1997geometrical} and the references therein). In particular, when $\Sigma=\SS^2$, the solution space of the $SU(3)$ Toda system coincides with the space of  holomorphic curves of $\mathbb{S}^2$ into $\mathbb{C}\mathbb{P}^2$~\cite{LWY2021Cla}. In mathematical physics,  Toda systems arise naturally in the study of non-abelian Chern–Simons gauge theories (see~\cite{dunne1991self,dunne1995self,nolasco1999double,nolasco2000vortex,yang1999relativistic,yang2001solitons}, and the reference therein).
	
	Since the solution space of  \eqref{eq:toda_mfN} is invariant under the addition of constants, we consider weak solutions in the subspace of the $N$-fold product of the Sobolev space $H^1(\Sigma)$ with zero average. Specifically, we define
	$$
	\mathcal{H} := \underbrace{\oH \times \dots \times \oH}_{N \text{ times }},
	$$
	where
	$
	\oH := \left\{ u \in H^1(\Sigma) : \int_\Sigma u \, dv_g = 0 \right\}.
	$
	
	Toda system has been widely studied in planar domains with Dirichlet boundary conditions and on closed Riemann surfaces. 
	These studies focus on various aspects of the system, including existence, uniqueness, and blow-up behavior. For existence results, we refer the reader to~\cite{jostandwang2001,jost_lin_wang2006,li2005solutions,battaglia_jevnikar_ruiz2015,Malchiodi2007SomeER,jevnikar_kallel_Malchiodi2015} and references therein. For uniqueness and non-degeneracy results, see~\cite{Bartolucci2023Non,LWY2021Cla}, and for blow-up analysis, we refer to~\cite{jost_lin_wang2006,LWZ2016Local,LWZ2025Cla,lee_degree_2018,musso_new_2016,D'aprile_asymmetric_2016}, among others.
	
	The  Toda system under Neumann boundary conditions on Riemann surfaces has been much less studied. In~\cite{zhu2011solutions}, X. B. Zhu constructs solutions to the $SU(3)$ Toda system as minimizers of the corresponding Euler-Lagrange functional in the case $\rho_i = 2\pi$ for all $i = 1, 2$.

	Let the local limit mass $\sigma_i(x')$ be defined by
	\[
	\sigma_i(x') = \lim_{n \to \infty} \lim_{r \to 0} \int_{\substack{x \in \Sigma, d_g(x,x') < r}} \frac{\rho_i^n V_i e^{u_i^n}}{\int_\Sigma V_i e^{u_i^n} dv_g} \, dv_g,
	\]
	for $i = 1, \ldots, N$, where $d_g(\cdot, \cdot)$ denotes the geodesic distance on $\Sigma$ with respect to the metric $g$.
	
	We are particularly interested in   the $SU(3)$ Toda system:
	\begin{equation}\label{eq:A_2_Toda}
		\left\{
		\begin{aligned}
			-\Delta_g u_1 &= 2\rho_1\left( \frac{V_1 e^{u_1}}{\int_{\Sigma} V_1 e^{u_1} \, dv_g} - 1\right) - \rho_2\left( \frac{V_2 e^{u_2}}{\int_{\Sigma} V_2 e^{u_2} \, dv_g} - 1\right) & &\text{ in } \intsigma, \\
			-\Delta_g u_2 &= 2\rho_2\left( \frac{V_2 e^{u_2}}{\int_{\Sigma} V_2 e^{u_2} \, dv_g} - 1\right) - \rho_1\left( \frac{V_1 e^{u_1}}{\int_{\Sigma} V_1 e^{u_1} \, dv_g} - 1\right) && \text{ in } \intsigma, \\
			\partial_{\nu_g} u_1 &= \partial_{\nu_g} u_2 = 0 & &\text{ on } \partial \Sigma.
		\end{aligned}
		\right.
	\end{equation}
	By the classical blow-up analysis, for any family of blow-up solutions $u^n := (u_1^n, u_2^n)$ to the $SU(3)$ Toda system with parameters $\rho^n := (\rho_1^n, \rho_2^n) \to (\rho_1,\rho_2)$, the set of blow-up points is given by
	\[
	\mathcal{S} := \Big \{ x \in \Sigma : \exists x_n \to x \text{ such that } \max_i \Big(u_i^n(x_n) -\int_{\Sigma} V_i e^{u_i^n} dv_g\Big)  \to \infty \Big\}.
	\]
	This set is finite and consists of points with local limit masses $(\sigma_1(x), \sigma_2(x))$ that take  one of the following values:
	\[
	\Big\{
	\Big( 0, \frac{1}{2} \varrho(x) \Big), \Big( \frac{1}{2} \varrho(x), 0 \Big), \Big( \frac{1}{2} \varrho(x), \varrho(x) \Big), \Big( \varrho(x), \frac{1}{2} \varrho(x) \Big), \Big( \varrho(x), \varrho(x) \Big)
	\Big\},
	\]
	for $x \in \mathcal{S}$, where $\varrho(x) = 8\pi$ if $x \in \Sigma$ and $\varrho(x) = 4\pi$ if $x \in \partial \Sigma$.
	
	The above result was established in~\cite{jostandwang2001,jost_lin_wang2006} for interior blow-up points. The boundary case is expected to follow by a similar blow-up analysis, with half-mass contributions arising from boundary effects. Since this lies outside the scope of the present paper, we omit further details.

	According to the value of the local limit masses, we classify the blow-up phenomena of the $SU(3)$ Toda system around the blow-up points into the following three scenarios:
	\begin{itemize}
		\item \textbf{Partial blow-ups}: $\left(0, \frac{1}{2} \varrho(x)\right), \left(\frac{1}{2} \varrho(x), 0\right)$;
		\item \textbf{Asymmetric blow-ups}: $ \left(\frac{1}{2} \varrho(x), \varrho(x)\right), \left(\varrho(x), \frac{1}{2} \varrho(x)\right)$;
		\item \textbf{Full blow-ups}: $\left(\varrho(x), \varrho(x)\right)$.
	\end{itemize}
	
	Next, we introduce works related to the existence of blow-up solutions for the $SU(3)$ Toda system. 
	
	Partial blow-up solutions, under a  non-degeneracy condition, have been constructed in various settings: on planar domains under Dirichlet boundary condition in \cite{DAprile_Pistoia_Ruiz2015}, on closed surfaces in \cite{lee_degree_2018}, and on Riemann surfaces with boundary under Neumann boundary condition   in \cite{ABH2024Partial}.
	
	For asymmetric blow-ups, W. Ao and L. Wang in \cite{ao2014new} introduce a family of blow-up solutions for Toda systems with Dirichlet boundary conditions on a unit ball centered at the origin, exhibiting a single blow-up point at the center. Musso, Pistoia, and J. Wei introduce a so-called ``$k$-symmetric'' property for the planar domain  and construct blow-up solutions for a $SU(N+1)$ Toda system, applying singular perturbation methods in \cite{musso_new_2016}.  Meanwhile, D'Aprile, Pistoia, and Ruiz obtain a similar result to $SU(3)$ Toda system for planar domains with Dirichlet boundary conditions as $\rho_1 \to 8\pi$, while keeping the fixed parameter $\rho_2 \in (4\pi, 8\pi)$ in \cite{D'aprile_asymmetric_2016}. The ``$k$-symmetric'' condition is a technical assumption introduced to ensure non-degeneracy of the limiting linearized problem.
	
	For full blow-ups, the special case where $V_1 = V_2$, $u_1 = u_2$, and $\rho_1 = \rho_2$ reduces the system to a mean field equation. Blow-up solutions for this equation have been constructed in \cite{baraket_construction_1997, Esposito2005} for bounded domains in $\R^2$, in \cite{Bartolucci2020, Esposito2014singular, figueroa2022bubbling} for closed Riemann surfaces, and in \cite{HBA2024} for Riemann  surfaces with boundary.  For the general case, the full blow-up scenario remains largely unresolved. C.-S. Lin, J. Wei, and C. Zhao   analyzed fully blow-up solutions of the $SU(3)$ Toda system  and derived  necessary conditions for their existence, highlighting the substantial difficulties in constructing such solutions in \cite{LWZ12}. Further progress on fully blowup  solutions has been made in subsequent works, see for instance \cite{Ao16, LWZ2016Local,Zhang20} and the references therein.

	To extend the results in \cite{musso_new_2016, D'aprile_asymmetric_2016} to the case of Toda systems of rank $N$ equipped with homogeneous Neumann boundary conditions on Riemann surfaces with boundary, this paper aims to construct blow-up solutions of the  Toda system \eqref{eq:toda_mfN} with $\rho_i \to 2\alpha_i \pi m$  for $i = 1, \dots, N$ on ``$k$-symmetric'' surfaces with $k > \frac{1}{2} \alpha_N $ via singular perturbation methods, where 
	\[ \alpha_i=2i \text{ for }i=1,\dots, N-1, \text{ and }\alpha_N= \begin{cases}
		2 N & \text{ for }  \mathbf{A}_N\text{ or } \mathbf{B}_N \\
		4N-2 & \text{ for } \mathbf{C}_N\\
		8 & \text{ for } \mathbf{G}_2
	\end{cases}.\]
	See  \eqref{def:alpha_i}  in Section \ref{sec:app} for more details.
	
	We denote by $O(3)$,  
	the orthogonal group of degree 
	$3$ acting on  $\mathbb{R}^3$. 
	For any $k \in \mathbb{N}_+$, we define
	\begin{equation*}
		\fR_{k} =
		\begin{pmatrix}
			\cos (2\pi/k) & \sin (2\pi/k) & 0 \\
			-\sin(2\pi/k) & \cos (2\pi/k) & 0 \\
			0 & 0 & 1
		\end{pmatrix}
		\in O(3),
	\end{equation*}
	which represents a rotation by angle $\frac{2\pi} k$ around the $z$-axis.  
	
	\begin{definition}\label{def:surface}
		Let $\Sigma$ be a Riemann surface. For an integer $k\geq 1$, we say $\Sigma$ is {\it $k$-symmetric} if 
		$\Sigma$ can be embedded in $\R^3$ and is invariant under $\fR_{k}$. 
		We say $x$ is a {\it $k$-symmetric center} of $\Sigma$ if $x\in \Sigma_0:=\left\{x\in \Sigma: \fR_{k}^i(x)=x \text{ for any } i=1,2,3,\dots\right\}$.
	\end{definition}
	
	\begin{definition}
		Let $\Sigma$ be a $k$-symmetric surface. A function $f: \Sigma \rightarrow \mathbb{R}$ is said to be \emph{$\fR_k$-invariant} if 
		$
		f(x) = f(\fR_k x)  \text{ for all } x \in \Sigma.
		$
	\end{definition}
	
	For technical reasons, we assume throughout that $\Sigma$ is a $k$-symmetric surface with $k > \frac 1 2 \alpha_N$, which is given by \eqref{def:alpha_i} in Section \ref{sec:app}. Due to the smoothness of the boundary, it follows that 
	$
	\Sigma_0 \cap \partial \Sigma = \emptyset.
	$

	We are now ready to state the main result of this paper.
	\begin{theorem}\label{thm:main_asymmetric}
		Let $k >\frac 1 2 \alpha_N$, where $\alpha_i$ is defined by \eqref{def:alpha_i} for $i=1,\dots, N$, and assume that $\Sigma$ is a $k$-symmetric Riemann surface with smooth boundary. Suppose the potential functions $V_1, \dots, V_N$ are $\fR_k$-invariant, and 
		\[
		\Sigma_0:=\left\{x\in \Sigma: \fR_{k}^i(x)=x \text{ for any } i=1,2,3,\dots\right\}. 
		\]
		Then for any $m$ distinct points   $\xi^*_1, \dots, \xi^*_m $ in  $ \Sigma_0$, there exists a family of solutions $\bu_\varepsilon = (u_{1,\varepsilon}, \dots, u_{N,\varepsilon})$ to the  Toda system \eqref{eq:toda_mfN} such that:
		\begin{itemize}
			\item[i)] $\bu_\varepsilon$ blows up exactly at the points $\xi^*_1, \dots, \xi^*_m$ as $\varepsilon \to 0$;
			\item [ii)]the parameters $\rho^\varepsilon = (\rho_1^\varepsilon, \dots, \rho_N^\varepsilon)$ satisfy
			\[
			\rho^\varepsilon \to (2\alpha_1 \pi m,\dots,  2\alpha_N \pi m);
			\]
			\item[iii)] for each $i = 1, \dots, N$, the following  weak-$*$ convergence holds:
			\[
			\rho_i^\varepsilon \frac{V_i e^{u_{i,\varepsilon}}}{\int_\Sigma V_i e^{u_{i,\varepsilon}}  dv_g \, } dv_g   \stackrel{*}{\rightharpoonup}  \sum_{j=1}^m 2\alpha_i \pi \delta_{\xi^*_j},  \text{ as } \varepsilon \to 0,
			\]
		\end{itemize}
		where $\delta_{x}$ is the Dirac measure on $\Sigma$ concentrated at the point $x.$
	\end{theorem}
	
	\noindent
	As an immediate consequence, we obtain the following result for the $SU(3)$ Toda system.
	\begin{corollary}
		Let $k > 2$, and suppose that $\Sigma$ is a $k$-symmetric Riemann surface, and the potentials $V_1$ and $ V_2$ are $\fR_k$-invariant. Assume $\Sigma_0 \neq \emptyset$ and $m \leq \# \Sigma_0$, where 
		$\# \Sigma_0$
		denotes the cardinality of $\Sigma_0$. Then,  there exists a family of solutions to the system \eqref{eq:A_2_Toda} with $(\rho_1^{\varepsilon}, \rho_2^\varepsilon) \to (4\pi m, 8\pi m)$ or $(8\pi m, 4\pi m)$ that blows up at exactly $m$ distinct points in $\Sigma_0$, exhibiting asymmetric blow-ups.
	\end{corollary}

	\begin{remark}
		The $k$-symmetry assumption is restrictive, but it includes two standard geometric models: the sphere $\SS^2$ and the upper hemisphere $\SS^2_+$. 
		Our result applies to compact surfaces with boundary (hence including $\SS^2_+$), and the method also allows the case $\partial\Sigma=\emptyset$ (thus covering closed surfaces such as $\SS^2$).
		
		For these two model cases, the blow-up locations of the solutions constructed in Theorem~\ref{thm:main_asymmetric} can be described more explicitly. 
		\begin{itemize}
			\item Let $\Sigma=\SS^2$ and assume that $V_1$ and $V_2$ are constants.
			For any $\xi\in\SS^2$, there exists a family of blow-up solutions to the $SU(3)$ Toda system~\eqref{eq:toda_mfN} as $(\rho_1^\varepsilon,\rho_2^\varepsilon)\to(4\pi,8\pi)$ or $(8\pi,4\pi)$, blowing up at $\xi$.
			Moreover, there exists a family of blow-up solutions as $(\rho_1^\varepsilon,\rho_2^\varepsilon)\to(8\pi,16\pi)$ or $(16\pi,8\pi)$, blowing up at the antipodal pair $\{\xi,-\xi\}$, where $-\xi$ denotes the antipodal point of $\xi$.
			
			\item Let $\Sigma=\SS^2_+$ and assume that $V_1$ and $V_2$ are $\mathfrak{R}_3$-invariant.
			Then there exists a family of blow-up solutions to the $SU(3)$ Toda system~\eqref{eq:toda_mfN} as $(\rho_1^\varepsilon,\rho_2^\varepsilon)\to(4\pi,8\pi)$ or $(8\pi,4\pi)$, blowing up at the north pole $\xi=(0,0,1)$.
		\end{itemize}
		
	\end{remark}

	Following the approach in \cite{D'aprile_asymmetric_2016, musso_new_2016}, we construct blow-up solutions of \eqref{eq:toda_mfN} via singular perturbation methods that combine variational methods with a Lyapunov–Schmidt reduction. One key difference in our setting is that we study the Toda system on surfaces, requiring the use of isothermal coordinates to pull back standard bubbles from the plane. Moreover, the Neumann boundary condition poses additional challenges, as the maximum principle is no longer available. Instead, we employ Green representation formulas and $L^p$-estimates to analyze the projected bubbles.
	
	To perform the finite-dimensional reduction, we compute the kernel of the limiting linearized operator. Due to the higher-order blow-up we consider, the kernel contains non-radial elements that do not align with the tangent space of the approximation manifold. To overcome this, we introduce $k$-symmetry on the surface, inspired by \cite{musso_new_2016}, which addresses invertibility in a $k$-symmetric subspace and implies that  the finite-dimensional part is trivial.
	
	The main analytical difficulty lies in dealing with component interactions and the non-symmetry of the Cartan matrices $\mathbf{B}_N, \mathbf{C}_N$ and $\mathbf{G}_2$. We resolve this using localized estimates in annular regions and treating the last component separately. The resulting solutions exhibit blow-up at $k$-symmetric centers which are interior points only; boundary blow-up remains out of reach due to degeneracy in our setting, and is left for future study.
	
	\section*{Acknowledgements}
	We would like to thank the referee for  useful comments  towards improving the presentation of this paper.  The first-named author also gratefully acknowledges Prof. T. Bartsch and Prof. M. Musso for their insightful discussions.

	\subsection*{Notations}
	Throughout this paper, we use the terms ``sequence'' and ``subsequence'' interchangeably.  The constant denoted by $C$ in our deduction may assume different values across various equations or even within different lines of equations. We adopt the standard asymptotic notation $ \mathcal{O}(1) $ and $ o(1)$ to describe the behavior of  quantities. More precisely, given functions $ g$ and $ f$, the notation $ g = \mathcal{O}(f)$ indicates that $ |g/f| \leq C$ for some constant $ C > 0$, while $ g = o(f)$ means that $ g/f \to 0$ as $ \varepsilon \to 0$.
	We set that $\B_r(\hat{y}):=\{ y\in \R^2: |y-\hat{y}|^2< r^2\}$, $ \B_r= \B_r(\mathbf{0})$, and  
	$\B^+_r=\B_r\cap \{ y_2\geq 0\}$. 
	For any $f\in L^1(\Sigma)$,  $\overline{f}:=\int_{\Sigma} f  dv_g.$ 
	
	\vskip10pt
	The rest of this paper is organized as below. 
	In Section \ref{sec:2}, we present preliminary results, including isothermal coordinates,  Green’s functions, and  approximate solutions. Section \ref{sec:3} is devoted to proving the invertibility of the limiting linearized operator. In Section \ref{sec:4}, we address the nonlinear problem, providing estimates for the error terms and establishing solvability via the contraction mapping principle. Section \ref{sec:5} completes the proof of Theorem \ref{thm:main_asymmetric}. Finally, Appendix \ref{App:A} contains technical estimates, focusing on the asymptotic behavior of the projected bubbles.
	
	\section{Preliminary}\label{sec:2}
	We begin in Section \ref{sec:2} with some basic concepts, where we introduce isothermal coordinates,  Green’s functions related to Neumann boundary conditions, and construct approximate solutions.
	\subsection{Isothermal coordinates and Green functions}
	We begin by  introducing a family of isothermal coordinates (refer to \cite{chern1955, Esposito2014singular,yang2021125440}, for instance). For any $\xi\in\Sigma$,  there exists an isothermal coordinate system $\left(U(\xi), y_{\xi}\right)$ such that $y_{\xi}$  maps an open  neighborhood $U(\xi)$ around $\xi$  onto
	\[ B^{\xi}=\begin{cases}
		\B_{2r_{\xi}} & \text{ for } \xi\in \intsigma\\
		\B_{2r_{\xi}}^+ & \text{ for } \xi\in \Sigma
	\end{cases}, \]
	in which $g=\sum_{i=1}^2 e^{\hat{\varphi}_\xi(y_{\xi}(x))}  \mathrm{d} x^i \otimes \mathrm{d} x^i$ and  $y_{\xi}(\xi)=(0,0)$.  
	Here, the conformal factor  $\hat{\varphi}_\xi:B^\xi\to\R$ is related to the Gaussian curvature $K_g$ and the geodesic curvature $k_g$ of the Riemann surface $(\Sigma, g)$,  by following  equation
	\begin{equation}
		\label{eq:Gauss}
		-\Delta\hat{\varphi}_\xi(y) = 2K_g\big(y^{-1}_\xi(y)\big) e^{\hat{\varphi}_\xi(y)} \quad\text{for all } y\in B^\xi. 
	\end{equation}
	Additionally, for $\xi\in \partial\Sigma$, we have 
	\[ -\partial_{y_2} \hat{\varphi}(y)=2k_g\circ y_{\xi}^{-1}(y)  e^{\hat{\varphi}(y)/ 2}, \text{ for all } y\in B^{\xi}\cap \{ y_2=0\}\]
	and 
	$
	\left(y_{\xi}\right)_*(\nu_g(x))=\left. -e^{ -\frac{\hat{\varphi}_{\xi}(y)}2} \frac {\partial} { \partial y_2 }\right|_{	y=y_{\xi}(x)}.
	$
	For $\xi\in \Sigma$ and $0<r\le 2r_\xi$ we set
	\[
	B_r^\xi := B^\xi \cap\B_r\quad \text{and}\quad U_{r}(\xi):=y_\xi^{-1}(B_{r}^{\xi}).
	\]
	As stated in~\cite{HBA2024}, 
	$y_\xi$ and $\hat{\varphi}_\xi$ are assumed to depend smoothly on $\xi$ in any given isothermal chart. Additionally, $\hat{\varphi}_\xi$ satisfies $\hat{\varphi}_\xi(\mathbf{0})=0$ and $$\nabla\hat{\varphi}_\xi(\mathbf{0})=\begin{cases}
		(0,0)  &\text{ if }\xi\in \intsigma\\
		(0, -2k_g(\xi)) &\text{ if }\xi\in \partial\Sigma
	\end{cases},$$
	where $k_g$ is the geodesic curvature of the boundary $\partial\Sigma$.

	Let $\chi$ be a radial cut-off function in  $C^{\infty}(\mathbb{R}, [0,1])$ such that 
	\begin{equation*}
		\chi(s)=\left\{\begin{array}{ll}
			1,& \text{ if } |s|\leq 1\\
			0, &\text{ if } |s|\geq 2
		\end{array}\right..
	\end{equation*} 
	And for fixed $r_0<\frac  1 4 r_{\xi}$, we denote that  $\chi_{\xi}(x):= \chi( \frac{|y_{\xi}(x)|}{r_0})$ and $\varphi_{\xi}(x):= \hat{\varphi}_{\xi} ( y_{\xi}(x))$. 
	For any $ \xi \in \Sigma $, we define the Green function for the Laplace-Beltrami operator  with homogeneous  Neumann boundary condition by the following equations:
	\begin{equation*}
		\left\{\begin{array}{ll}
			-\Delta_g G^g(x,\xi)  = \delta_{\xi} - 1 & x\in \intsigma \\
			\partial_{ \nu_g } G^g(x,\xi) = 0 & x\in \partial \Sigma \\
			\int_{\Sigma} G^g(x,\xi)\, dv_g(x) = 0
		\end{array}\right.,
	\end{equation*}
	where $ \delta_{\xi} $ is the Dirac mass on $ \Sigma $ concentrated at $ \xi $.
	
	Let the function
	\begin{equation*}
		\Gamma^g_{\xi}(x)=\Gamma^g(x,\xi)= \left\{ \begin{array}{ll}	-\frac{1}{2\pi} \chi_{\xi}(x)\log{|y_{\xi}(x)|}& \text{ if }\xi\in \intsigma \\	-{\frac{1}{\pi}}\chi_{\xi}(x)\log{|y_{\xi}(x)|}	& \text{ if } \xi\in\partial\Sigma \end{array}\right..
	\end{equation*}  
	Decomposing the Green function  
	$ G^g(x,\xi)= \Gamma^g_{\xi}(x)+ H^g_{\xi}(x),$ we have the function $H^g_{\xi}(x):= H^g(x,\xi)$ solves the following equations: 
	\begin{equation}\label{eqR}
		\left\{\begin{array}{ccll}
			-\Delta_g H^g_\xi
			&=&\frac{4}{\varrho(\xi)} (\Delta_g \chi_{\xi}) \log \frac 1 {|y_{\xi}|}+ \frac{8}{\varrho(\xi)} \la\nabla\chi_{\xi}, \nabla\log \frac 1 {|y_{\xi}|}\ra_g 	- 1, & \text{ in }\intsigma\\
			\partial_{ \nu_g} H^g_\xi&=&- \frac{4}{\varrho(\xi)}(\partial_{ \nu_g} \chi_{\xi}) \log\frac{1}{|y_{\xi}|}-\frac{4}{\varrho(\xi)}\chi_{\xi}\partial_{ \nu_g}  \log\frac{1}{|y_{\xi}|}, &\text{ on }\partial \Sigma\\
			\int_{\Sigma} H^g_\xi \,dv_g&=& - \frac{4}{\varrho(\xi)}\int_{\Sigma}  \chi_{\xi}\log\frac{1}{|y_{\xi}|} dv_g &
		\end{array}\right..\end{equation}
	By the regularity of the elliptic equations (refer to \cite{Wehrheim2004,Agmon1959}), there exists a unique smooth solution  $H^g(x,\xi)$, which solves \eqref{eqR} in the H\"{o}lder space $C^{2,\alpha}(\Sigma)$ with $\alpha\in (0,1)$. 
	$H^g(x,\xi)$ is the regular part of  $G^g(x,\xi)$ and 
	$R^g(\xi):=H^g(\xi,\xi)$ is  Robin's function on $\Sigma$. It is clear that  to check  $H^g(\xi,\xi)$ is independent of the choice of the cut-off function $\chi$ and the local chart.
	
	\subsection{Approximation solutions}\label{sec:app}
	
	Assume that  $$ \frac{\rho_1}{\int_{\Sigma} V_1 e^{u_1} dv_g}=\dots= \frac{\rho_N} {\int_\Sigma V_N e^{u_N} dv_g}:=\varepsilon.$$ Then, the system \eqref{eq:toda_mfN} can be reformulated as a system of two coupled Liouville-type equations:
	\begin{equation}\label{eq:todaN}
		\begin{cases}-\Delta_g u_i= \varepsilon\sum_{j=1}^N a_{ij}(  V_ie^{u_j}- \int_\Sigma V_j e^{u_j} dv_g)  & \text { in } \intsigma \\ \partial_{\nu_g}u_i =0 & \text { on } \partial \Sigma \end{cases},
	\end{equation}	
	for $i=1,\dots, N.$
	
	We consider  $\Sigma$ is a ``$k$-symmetric'' surface with  $\Sigma_0\neq\emptyset$ and  $m\leq \#\Sigma_0$. 
	Given $\{\xi^*_1,\dots,\xi^*_m\}\subset \Sigma_0\cap \intsigma$, we will construct a family of solutions blowing up exactly at $\{\xi^*_1,\dots,\xi^*_m\}$. 
	To construct blow-up solutions of \eqref{eq:toda_mfN}, it is sufficient to consider the problem 
	\eqref{eq:todaN}. 
	
	For fixed $\xi^*=(\xi^*_1,\dots, \xi^*_m)\in \Xi_{m}:= \intsigma^m \setminus\{ \xi=(\xi_1,\dots,\xi_m)\in \Sigma^m: \exists \xi_j=\xi_{j'}\text{ for some } j\neq j'\}$, we are going to construct a family of blow-up solutions $u_{\varepsilon}=(u_{1,\varepsilon},\dots,  u_{N,\varepsilon})$ of \eqref{eq:todaN} as $\varepsilon\rightarrow 0$. 
	
	To construct the approximation, we introduce the following singular Liouville equations:  
	\begin{equation}\label{eq:anatz}
		\left\{ \begin{array}{lc}
			-\Delta w=|y|^{\alpha-2} e^{w}& \text{ in }\R^2\\
			\int_{\R^2} |y|^{\alpha-2} e^{w}<\infty&
		\end{array}
		\right.,
	\end{equation}
	where $\alpha\geq 2$.  For any $\alpha\geq 2$,  from the classification from Prajapat and Tarantello in \cite{PT2001}, the radially symmetric solutions of the singular Liouville problem~\eqref{eq:anatz} is as follows:
	\[
	w_{\tau}^{\alpha}(y) := \log \frac{2\alpha^2 \tau^{\alpha}}{(\tau^{\alpha} + |y|^{\alpha})^2} \quad y \in \R^2, \tau > 0.
	\]
	Moreover, we have
	$
	\int_{\R^2} |y|^{\alpha-2} e^{w_{\tau}^{\alpha}(y)}  dy = 4\pi\alpha,
	$
	by direct calculation. 
	For any $\xi\in \Sigma$, applying isothermal coordinate $(y_{\xi}, U(\xi))$, we can pull-back $w^{\alpha}_{\tau}$ to the Riemann surface around $\xi$,
	\[ U^{\alpha}_{\tau,\xi}:= w^\alpha_\tau \circ y_{\xi}(x)=\log \frac{2\alpha^2 \tau^{\alpha}}{(\tau^{\alpha} + |y_{\xi}(x)|^{\alpha})^2} \text{ in } U(\xi).  \]
	Then we project the local bubbles into the functional space $\oH$ by following equations:
	\begin{equation*}
		\left\{\begin{array}{lc}
			-\Delta_gPU^{\alpha}_{\tau,\xi} = \chi_{\xi} e^{-\varphi_{\xi}} |y_{\xi}|^{\alpha-2} e^{U^{\alpha}_{\tau,\xi}}-\overline{\chi_{\xi} e^{-\varphi_{\xi}} |y_{\xi}|^{\alpha-2} e^{U^{\alpha}_{\tau,\xi}}} &\text{ in } \intsigma\\
			\partial_{\nu_g} PU^{\alpha}_{\tau,\xi}=0 &\text{ on } \partial\Sigma\\
			\int_{\Sigma} PU^{\alpha}_{\tau,\xi} dv_g=0&
		\end{array}\right..
	\end{equation*}
	We take  for any $ i\in \{1,\dots, N-1\}$
	\begin{equation}
		\label{def:alpha_i}
		\begin{split}
			\alpha_i&=2i 
			\text{ and }
			\alpha_N= 2-2(N-1)a_{N N-1}=\begin{cases}
				2 N & \text{ for }  \mathbf{A}_N\text{ or } \mathbf{B}_N \\
				4N-2 & \text{ for } \mathbf{C}_N\\
				8& \text{ for } \mathbf{G}_2
			\end{cases},
		\end{split} 
	\end{equation}
	which satisfies that  for any $i=1,\dots, N$, 
	$\alpha_i-2=-\sum_{i'<i} a_{ii'} \alpha_i'. $
	The concentration parameter 
	\begin{equation}
		\label{def:delta_i} \delta_{N}= \varepsilon^{\frac 1 {\alpha_{N}}} \text{ and } 
		\delta_{i}=\begin{cases}
			\varepsilon^{\frac{N+1-i}{\alpha_i}}  & \text{ for } \mathbf{A}_N, \mathbf{C}_N \text{ or } \mathbf{G}_2 \\
			\varepsilon ^{\frac{N+2-i}{\alpha_i}} & \text{ for } \mathbf{B}_N
		\end{cases},
	\end{equation}
	which satisfies that 
	$  \delta_i^{\alpha_i} \Pi_{i'>i}  \delta_{i'}^{ a_{ii'} \alpha_{i'}}= \varepsilon.$ 
	It is easy to see 
	\begin{equation}
		\label{eq:diff_deltai} \frac{\delta_{i-1}}{\delta_i}= \varepsilon^{ \frac{N+1}{2(i-1)i}} \text{ for } i=2, \dots, N-1 \text{ and } 
		\frac{\delta_{N-1}}{\delta_N}=\cO (\varepsilon^{\frac 1 {2(N-1)}}).
	\end{equation}
	Set $\delta_{i,j}:=d_{i,j}\delta_i$ for $i=1,\dots, N$ and $j=1,\dots,m$, 
	where $d_{i,j}>0$ solves the following identities for
	\begin{align}\label{eq:def_d_ij_as}
		&	\alpha_i\log d_{i,j} +\sum_{i'>i} a_{ii'}\alpha_{i'} \log d_{i',j}= -2\log\alpha_i\\
		& +\frac{1}{2}\Big(\alpha_i +\sum_{\sumi}^N \frac {a_{ii'}} 2  \alpha_{i'}\Big)  \cdot \Big(\varrho(\xi_j)R^g(\xi_j)+\sum_{j'\neq j}\varrho(\xi_{j'})G^g(\xi_{j'},\xi_j)\Big)+\log V_i(\xi_j). 
		\nonumber 
	\end{align}
	Let 
	\begin{equation}
		\label{eq:def_A} 
		\cA_{ij}=\left\{x\in\Sigma: \sqrt{\delta_{i,j}\delta_{i-1,j}}\leq |y_{\xi_j}(x)|\leq \sqrt{\delta_{i,j}\delta_{i+1,j}}\right\}, 
	\end{equation}
	for  $i=1,\dots, N$ and $ j=1,\dots , m, $ 	where  $\delta_{0,j}:=0$ and $\delta_{N+1,j}:=\infty$.

	For $\tau_0>0$, the isothermal charts $\{ (y_{\xi_j}, U(\xi_j)):\sum_j  d_g(\xi_j,\xi_j^*)<\tau_0\}$ satisfying that  for some uniform radius $r_0>0$,  $$U_{4r_0}(\xi_j)\subset  U(\xi_j)\subset \subset \Sigma,$$  
	$U_{4r_0}\cap \partial\Sigma=\emptyset$ and 
	$U_{4r_0}(\xi_{j'})\cap U_{4r_0}(\xi_j)=\emptyset$ for all $j\neq j'.$
	
	For simplicity, we define
	$U^i_j= U^{\alpha_i}_{\xi_j,\delta_{i,j}}, PU^i_j=PU^{\alpha_i}_{\xi_j,\delta_{i,j}}, \chi_j=\chi(\frac{|y_{\xi_j}|}{r_0})$ and $\varphi_j=\hat{\varphi}_{\xi_j}(y_{\xi_j})$.
	The approximation solution $\bW_{\varepsilon}=(W_{1,\varepsilon},\dots, W_{N,\varepsilon})$ is defined by 
	\begin{equation*}
		W_{i,\varepsilon}=\sum_{j=1}^m PU^i_j +\sum_{\sumi }^N\sum_{j=1}^m\frac{a_{ii'}}2 PU^{i'}_j,\text{ for } i=1,\dots, N.
	\end{equation*}
	We define that 
	$$\cH_{k}:=\left\{u=(u_1,\dots, u_N)\in \cH: u_i \text{ is } \fR_{k}\text{-invariant for } i=1,\dots, N \right\}. $$
	Next, we are going to construct  solutions with the form
	$ \bu_{\varepsilon}=\bW_{\varepsilon}+\bphi_{\varepsilon},$
	where $\bphi_{\varepsilon}=(\phi_{1,\varepsilon}, \dots,\phi_{N,\varepsilon})\in \cH_{k}$ is the error term. 
	
	\section{The limiting linearized operator}
	\label{sec:3}

	Section \ref{sec:3} focuses on the analysis of the limiting linearized operator and establishes its invertibility in the space $\mathcal{H}_k$. While the general strategy follows the standard approach introduced in \cite{Esposito2005,del_pino_singular_2005}, the argument becomes  simpler in our setting due to the imposed $k$-symmetry of the surface. Nevertheless, for completeness and to clarify where the $k$-symmetric property plays a role, as well as how the interaction between different components is treated, we provide a detailed proof.

	We consider the following linear operator associated with the problem~\eqref{eq:todaN}: 
	\begin{equation}
		\label{eq:linear_op_a}  \cL_{\xi,\varepsilon}(\bphi):=( L^1_{\xi,\varepsilon}(\bphi),\dots, L^N_{\xi,\varepsilon}(\bphi) ),
	\end{equation}
	where for any $i=1, \dots, N $
	\begin{equation*}
		\begin{split}
			L^i_{\xi,\varepsilon}(\bphi):&= -\Delta_g\phi_i -\sum_{j=1}^m\left( \chi_j e^{-\varphi_j} |y_{\xi_j}|^{\alpha_i-2} e^{U^i_j} \phi_i-\overline{\chi_j e^{-\varphi_j} |y_{\xi_j}|^{\alpha_i-2} e^{U^i_j} \phi_i}\right)\\
			&- \sum^N_{\sumi
			} \sum_{j=1}^m\frac{a_{ii'}}2\left( \chi_j e^{-\varphi_j} |y_{\xi_j}|^{\alpha_{i'}-2} e^{U^{i'}_j} \phi_{i'}-\overline{\chi_j e^{-\varphi_j} |y_{\xi_j}|^{\alpha_{i'}-2} e^{U^{i'}_j} \phi_{i'}}\right).
		\end{split}
	\end{equation*}
	The key lemma is  the non-degeneracy of the linear operator $\cL_{\xi,\varepsilon}$. Formally, for $i=1,\dots, N, j=1,\dots,m$, we can derive the local limit operator of $L^i_{\xi,\varepsilon}$  is 
	$$-\Delta \phi -2\alpha_i^2 \frac{|y|^{\alpha_i-2}}{(1+|y|^{\alpha_i})^2}\phi$$ 
	by a proper scaling around $\xi_j$ on an isothermal chart (refer to Lemma~\ref{lem:invertible_as}). Del Pino et al. in \cite{DelPino2012Nondegeneracy} prove that the kernel space is generated by (in polar coordinate $(r,\theta)$) 
	\[\phi^0(r,\theta):=\frac{1-|r|^{\alpha_i}}{1+|r|^{\alpha_i}}, \, \phi^1(r,\theta):=\frac{|r|^{\frac{\alpha_i}{2}}}{1+|r|^{\alpha_i}}\cos \frac{\alpha_i}{2}\theta , \,  \phi^2(r,\theta):=\frac{|r|^{\frac{\alpha_i}{2}}}{1+|r|^{\alpha_i}} \sin \frac{\alpha_i}{2} \theta. \] 
	The $k$-symmetric condition of $\cH_{k}$ with $k>\frac {\alpha_{N}} 2$ excludes $\phi^1$ and $\phi^2$. To obtain the invertibility of the linearized operator $\cL_{\xi,\varepsilon}$, we need to lay out $\phi^0$ by introducing a family of test functions in $\oH$. 
	Let $z_i(y)=\frac {1-|y|^{\alpha_i}}{ 1+|y|^{\alpha_i}}.$
	It is easy to see that 
	$$-\Delta z_i(y)= 2\alpha_i^2 \frac{|y|^{\alpha_i-2}}{(1+|y|^{\alpha_i})^2}z_i(y), y\in\R^2. $$
	For any $j=1,\dots,m$, we define 
	\[ Z_{ij}(x)=\left\{ \begin{array}{ll}
		z_i\left(\frac{|y_{\xi_j}(x)| } {\delta_{i,j}} \right), & x\in U(\xi_j)\\
		0, & x\in \Sigma\setminus U(\xi_j)
	\end{array}\right.. \]
	Then, we project $Z_{ij}$ into the space $\oH$ by following equations: 
	\begin{equation}
		\label{eq:proj_Z}
		\left\{\begin{array}{ll}
			-\Delta_gPZ_{ij}= \chi_j e^{-\varphi_j} |y_{\xi_j}|^{\alpha_i-2} e^{U^i_j}Z_{ij}-\overline{\chi_j e^{-\varphi_j} |y_{\xi_j}|^{\alpha_i-2} e^{U^i_j}Z_{ij}} &\text{  in } \intsigma\\
			\partial_{\nu_g} PZ_{ij}=0 &\text{ on } \partial\Sigma\\
			\int_{\Sigma} PZ_{ij} dv_g=0&
		\end{array}\right..
	\end{equation}
	We define that 
	$$\fL^{0}_{k}:=  \left\{ h=(h_1,\dots, h_N):   \int_{\Sigma} h_i  dv_g =0\text{ and } h_i \text{ is }\fR_{k}\text{ invariant} \right\}. $$
	\begin{lemma}
		\label{lem:invertible_as}
		For any $p>1$, 
		there exist $\varepsilon_0 > 0$ and $C > 0$ such that for any $\varepsilon \in (0, \varepsilon_0)$,  $\bh=(h_1,\dots, h_N) \in (L^p(\Sigma))^N\cap\fL^0_{k}$ there exists    $\bphi=(\phi_1,\dots,  \phi_N) \in (W^{2,p}(\Sigma))^N \cap \cH_{k}$ as a unique solution of 
		\begin{equation}
			\label{eq:linear_key_a} \left\{ \begin{array}{ll}
				\cL_{\xi,\varepsilon}(\bphi)=\bh & \text{ in }\intsigma\\
				\partial_{\nu_g} \bphi =0 & \text{ on  }\partial\Sigma\\
				\int_{\Sigma}\bphi dv_g=0 
			\end{array}\right.,
		\end{equation}
		satisfying that  
		$$
		\|\bphi\| \leq C |\log \varepsilon| \|\bh\|_p,
		$$
		where $\|\bphi\|:= \sqrt{\sum_i \|\phi_i\|^2} \text{ and } \|\bh\|_p=\sum_i \|h_i\|_{p}$.
	\end{lemma}
	
	\begin{proof}
		We prove it by contradiction. Suppose Lemma~\ref{lem:invertible_as} fails, i.e. there exist $p>1$ and a sequence of $\varepsilon_n\rightarrow 0$ and $\bh_n:=(h_{1,n},\dots,  h_{N,n})\in (L^p(\Sigma))^N\cap \fL^0_{k}$ and $\bphi_n:=(\phi_{1,n},\dots, \phi_{N,n})\in  (W^{2,p}(\Sigma))^N\cap \cH_{k}$ solves ~\eqref{eq:linear_key_a} for $\bh_n$ satisfying  
		\[ \|\bphi_n\|=1\text{ and } |\log \varepsilon_n|\|\bh_n\|_p:=|\log\varepsilon_n|\sum_{i=1}^N \|h_{i,n}\|_{p}\rightarrow 0, \]
		as $n\rightarrow +\infty.$
		For simplicity, we still use the notations $\phi_i,$ $ h_i,$ $\varepsilon$ instead of $\phi_{i,n},$ $h_{i,n},$ $ \varepsilon_n$ for $i=1,\dots, N$. 
		We define that for $i=1,\dots, N, j=1,\dots,m$
		\[ \tilde{\phi}_{ij}(y)= \left\{\begin{aligned}
			&	\chi\left( \frac { \delta_{i,j}|y|}{r_0}\right)	\phi_i\circ y_{\xi_j}^{-1}(\delta_{i,j} y),&  &y \in \Omega_{ij}:= \frac{1}{\delta_{i,j}} B^{\xi_j}\\
			&	0& & y\in\R^2\setminus\Omega_{ij}
		\end{aligned}\right.. \]
		Then we consider the following spaces for  $\alpha\geq 2 $ and $\xi\in\Sigma$,
		$$ \rL^{\alpha}:=\Big\{ u: \Big\| \frac{|y|^{\frac{\alpha-2}{2}}}{1+|y|^{\alpha}} u \Big\|_{2} <+\infty\Big\}$$
		and 
		$$ \rH^{\alpha}:=\Big\{u: \|\nabla u\|_{2}+\Big\| \frac{|y|^{\frac{\alpha-2}{2}}}{1+|y|^{\alpha}} u \Big\|_{2}<+\infty \Big\}.$$
		\begin{itemize}
			\item[{\bf Step 1.}]\label{item:step1_a} {\it $\tilde{\phi}_{ij}\rightarrow c_{ij} \frac{1-|y|^{\alpha_i}}{1+|y|^{\alpha_i}}$ as $\varepsilon\rightarrow 0$ for some $c_{ij}\in\R$, which is weakly in $\rH^{\alpha_i}$ and strongly in $\rL^{\alpha_i}$.}
		\end{itemize}
		We fix an arbitrary $l=1,\cdots, N$. 
		The Sobolev inequality and H\"{o}lder's inequality yield that
		\begin{equation}\label{eq:h_phi}
			\Big| \int_{\Sigma} h_i \phi_l\Big|\leq \|\bh\|_p \|\phi_l\|_{p'}\leq \|\bh\|_p \|\phi_l\|=o\Big(|\log\varepsilon|^{-1}\Big)\rightarrow 0,
		\end{equation} 
		where $p, p'>1$ with $\frac 1 p+\frac 1 {p'}=1. $ 
		Moreover, $|\la \phi_i, \phi_l\ra|\leq C \|\phi_i\|\|\phi_l\|=\cO(1)$.

		Integrating \eqref{eq:linear_key_a} with test function $\phi_l$, by \eqref{eq:h_phi} and the definition $(a_{ij})$, we inductively
		obtain that  for $ i=1,\dots, N-1 $
		\begin{align}\label{eq:lem3.1-8}
			\sum_{j=1}^m \int_{\Sigma} \chi_je^{-\varphi_j} |y_{\xi_j}|^{\alpha_{i}-2} e^{U^{i}_j} \phi_i  \phi_l dv_g&= i \sum_{j=1}^m \int_{\Sigma} \chi_je^{-\varphi_j} |y_{\xi_j}|^{\alpha_{1}-2} e^{U^{1}_j} \phi_1  \phi_l dv_g +\cO(1),  \\
			\sum_{j=1}^m \int_{\Sigma} \chi_je^{-\varphi_j} |y_{\xi_j}|^{\alpha_{N}-2} e^{U^{N}_j} \phi_N  \phi_l dv_g&= \frac{N}{-a_{N-1 N}}\sum_{j=1}^m \int_{\Sigma} \chi_je^{-\varphi_j} |y_{\xi_j}|^{\alpha_{1}-2} e^{U^{1}_j} \phi_1  \phi_l dv_g
			\nonumber	\\
			&	+\cO(1). \nonumber
		\end{align}
		Using $\phi_l$  as a test function of 
		$L^N_{\xi,\varepsilon}(\bphi)=h_N$, we derive that 
		\begin{align}\label{eq:lem3.1-9}
			\cO(1)&= 	2\sum_{j=1}^m  \int_{\Sigma} \chi_je^{-\varphi_j} |y_{\xi_j}|^{\alpha_{N}-2} e^{U^{N}_j} \phi_N  \phi_l dv_g\\
			&\quad +  a_{N N-1}\sum_{j=1}^m \int_{\Sigma} \chi_je^{-\varphi_j} |y_{\xi_j}|^{\alpha_{N-1}-2} e^{U^{N-1}_j} \phi_{N-1}  \phi_l dv_g\nonumber
			\\
			&\stackrel{\eqref{eq:lem3.1-8}}{=}\frac{2N - a_{N-1 N} a_{N N-1} (N-1)}{-a_{N-1 N}}  \sum_{j=1}^m \int_{\Sigma} \chi_je^{-\varphi_j} |y_{\xi_j}|^{\alpha_{1}-2} e^{U^{1}_j} \phi_{1}  \phi_l dv_g+\cO(1).\nonumber
		\end{align}
		Since  $2N - a_{N-1 N} a_{N N-1} (N-1)\neq 0$ for all $N\geq 2$, we deduce that  
		\begin{align}	\label{eq:bouned_exp_phi_a}
			\sum_{j=1}^m \int_{\Sigma} \chi_je^{-\varphi_j} |y_{\xi_j}|^{\alpha_{l}-2} e^{U^{l}_j} \phi_l^2dv_g&=   l \int_{\Sigma} \chi_je^{-\varphi_j} |y_{\xi_j}|^{\alpha_{1}-2} e^{U^{l}_j} \phi_{1}  \phi_l dv_g \\
			&\stackrel{\eqref{eq:lem3.1-9}}{=}\cO(1),\nonumber
		\end{align}
		for $l=1,\dots, N.$
		By a straightforward calculation, for all $i=1,\dots, N$ 
		\[\sum_{j=1}^m \int_{\R^2}\frac{|y|^{\alpha_i-2}}{(1+|y|^{\alpha_i})^2}(\tilde{\phi}_{ij}(y))^2dy  =	\sum_{j=1}^m \int_{\Sigma} \chi_j e^{-\varphi_j} |y_{\xi_j}|^{\alpha_i-2} e^{U^i_j} \phi_i^2  dv_g\stackrel{\eqref{eq:bouned_exp_phi_a}}{=}\mathcal{O}(1).\]
		Additionally, by the assumption $\|\bphi\|=1$, we immediately  have 
		\[ \int_{\R^2} |\nabla \tilde{\phi}_{ij}|^2 \leq \int_{\Sigma} |\nabla\phi_i|_g^2  dv_g\leq 1.  \]
		It follows that $\tilde{\phi}_{ij}$ is uniformly bounded in $\rH^{\alpha_i}$. Applying~\cite[Proposition 6.1]{GP2013_Sinh},  $\rL^{\alpha_i}\hookrightarrow \rH^{\alpha_i}$ is {a compact embedding.}  Then  up to a subsequence, we have
		$\tilde{\phi}_{ij}  \rightarrow \tilde{\phi}_{ij}^0 $
		weakly convergent in $\rH^{\alpha_i}$ and strongly in $\rL^{\alpha_i} $. 
		For any $q>1$, 
		\begin{equation}
			\label{eq:L_p_exp_a}
			\int_{\Sigma} \left|\chi_j e^{-\varphi_j} |y_{\xi_j}|^{\alpha_i-2} e^{U^i_j}  \right|^q dv_g = \int_{U_{2r_0}(\xi_j)}\chi^q\Big(\frac{|y|}{r_0}\Big) \delta_{i,j}^{2-2q} \frac{ |y|^{q(\alpha_i-2)}}{(1+|y|^{\alpha_i})^2} dy =\mathcal{O}(\delta_{i,j}^{2(1-q)}). 
		\end{equation}
		By changing variables of~\eqref{eq:linear_key_a}, the estimate \eqref{eq:L_p_exp_a} implies  that 
		\begin{equation}\label{eq:lin_pre_limit}
			\left\{    \begin{array}{ll}
				-\Delta\tilde{\phi}_{ij}(y) = 2\alpha^2_i \frac{|y|^{\alpha_i-2}}{( 1+|y|^{\alpha_i})^2}\tilde{\phi}_{ij}+\psi_{ij} & \text{ in } \Omega_{ij}\\
				\tilde{\phi}_{ij} =0&\text{ on } \partial \Omega_{ij}
			\end{array}\right.,
		\end{equation}
		where 	for $q\in (1,\frac 4 3)$,
		\begin{align*}
			\psi_{ij}&= - \sum_{\sumi}^N 2\alpha_{i'}^2 \frac {a_{ii'}} 2 \frac{\delta_{i,j}^{\alpha_{i'}}\delta_{i',j}^{\alpha_{i'}}|y|^{\alpha_{i'}-2}}{(\delta_{i',j}^{\alpha_{i'}}+\delta_{i,j}^{\alpha_{i'}}|y|^{\alpha_{i'}})^2}\tilde{\phi}_{i'j}\Big( \frac{\delta_{i,j}}{\delta_{i',j}} y\Big)+\delta_{i,j}^2 h_i\circ y_{\xi_j}^{-1}(\delta_{i,j} y)\\
			&+ \mathcal{O}(\delta_{i,j}^2\sum_{i'=1}^N \delta_{i',j}^{\frac 2 q-2}),
		\end{align*}

		Next, we are going to show that 
		$\tilde{\phi}_{ij}$ converges to the solution of 
		\begin{equation}
			\label{eq:limit_linear_a}
			\left\{\begin{array}{ll}
				-\Delta \phi=2 \alpha_i^2 \frac{|y|^{\alpha_i-2}}{\left(1+|y|^{\alpha_i}\right)^2} \phi & \text { in }\R^2\\
				\int_{\R^2}|\nabla \phi(y)|^2 d y<+\infty &
			\end{array}\right.,
		\end{equation}
		in the distribution sense. 
		We take an  arbitrary $\varphi\in C^{\infty}_c(\R^2)$
		with  $\supp(\varphi)\subset \frac{\cA_{ij}}{\delta_{i,j}}=\left\{y\in \Omega_{ij}:  \sqrt{\delta_{i-1,j}/\delta_{i,j}}\leq |y|< \sqrt{\delta_{i+1,j}/\delta_{i,j}}\right\},$
		for $\varepsilon>0$ sufficiently small,	where $\cA_{ij}$ is defined by~\eqref{eq:def_A},  $\delta_{0,j}=0$ and $\delta_{N+1,j}=+\infty.$ \\
		For any $i\neq i'$ , 
		by direct calculation, it holds: 
		\begin{align}\label{eq:lem3.1-1}
			\int_{\cA_{ij}}\Big|\frac{\delta_{i',j}^{\alpha_{i'}}|y|^{\alpha_{i'}-2}}{(\delta_{i',j}^{\alpha_{i'}}+ |y|^{\alpha_{i'}})^2 }  \Big|^p  dy &= 
			\int_{ \frac 1 {\delta_{i',j}} \cA_{ij}} \delta_{i', j}^{2-2p}   \Big|\frac{|y|^{\alpha_{i'}-2}}{(1+ |y|^{\alpha_{i'}})^2 }  \Big|^p  dy\\
			&= \begin{cases}
				\cO \Big( \int_{|y|\leq \frac{\sqrt{\delta_{i,j}\delta_{i+1,j}}}{\delta_{i',j}} }  \delta_{i',j}^{2-2p} |y|^{(\alpha_{i'}-2)p} dy \Big)	   & \text{ if }i'> i\\
				\cO\Big( \int_{|y|\geq  \frac{\sqrt{\delta_{i,j}\delta_{i-1,j}}}{\delta_{i',j}}} \delta_{i',j}^{2-2p} |y|^{-(\alpha_{i'}+2)p}dy\Big)    & \text{ if }i'<i
			\end{cases} \nonumber\\ 
			&=\begin{cases}
				\cO(\delta_{i+1}^{1-p}\delta_{i}^{1-p}( \delta_{i}/\delta_{i+1})^{\frac{\alpha_{i'}p}{2}})	 & i'> i\\
				\cO( \delta_{i-1}^{1-p}\delta_{i}^{1-p}( \delta_{i-1}/\delta_i)^{\frac{\alpha_{i'}p}{2}})	& i'<i
			\end{cases}\stackrel{\eqref{eq:diff_deltai}}{=} o(1),\nonumber
		\end{align}
		for $p\in (1,2)$ sufficiently close to $1.$
		Applying $\varphi$ as a test function for~\eqref{eq:lin_pre_limit}, we deduce that 
		\begin{align*}
			&\quad 	\int_{\supp(\varphi)} \nabla \tilde{\phi}_{ij}\cdot\nabla \varphi - 2\alpha_i^2 \int_{\supp(\varphi)} \frac{|y|^{\alpha_i-2}}{(1+|y|^{\alpha_i})^2} \tilde{\phi}_{ij} \varphi = \int_{\supp(\varphi)} \psi_{ij} \varphi\\
			&= \mathcal{O}\Big(\int_{\cA_{ij}}  \sum_{\sumi}^N \frac {-a_{ii'}} 2 2\alpha_{i'}^2 e^{-\varphi_j} \frac{\delta_{i',j}^{\alpha_{i'}} |y_{\xi_j}|^{\alpha_{i'}-2}}{(\delta_{i',j}^{\alpha_{i'}}+|y_{\xi_j}|^{\alpha_{i'}})^2}\phi
			_{i'}\circ y_{\xi_j}^{-1}(y) \varphi(\delta_{i,j}^{-1} y)  dy\Big) \\
			&\quad+\int_{ \delta_{i,j}\supp(\varphi)} h_i\circ y_{\xi_j}^{-1}(y)  \varphi(\delta_{i,j} ^{-1}y) dy+ o(1)\\
			&= \mathcal{O}\Big(\Big(\int_{ \sqrt{\delta_{i-1,j}\delta_{i,j}}\leq |y|< \sqrt{\delta_{i,j}\delta_{i+1,j}}} \Big(\sum_{\sumi}^N  \frac{2\alpha_{i'}^2 \delta_{i',j}^{\alpha_{i'}} |y|^{\alpha_{i'}-2}}{(\delta_{i',j}^{\alpha_{i'}}+|y|^{\alpha_{i'}})^2}\Big)^q dy\Big)^{\frac 1 q }\|\phi_{i'}\|\Big) \\
			&\quad+\mathcal{O}\Big(  \Big(\int_{ \delta_{i,j}\supp(\varphi)}|  \varphi(\delta^{-1}_{i,j} y)|^{p'} dy \Big)^{\frac 1 {p'}} \|h_i\|_{p} \Big)+ o(1)\stackrel{\eqref{eq:lem3.1-1}\text{ and }\eqref{def:delta_i}}{=}o(1),
		\end{align*}
		where $\frac 1 p+\frac 1 {p'}=1, q>1$ sufficiently close to $1$. 
		Since $\cA_{ij}$ invades the whole space $\R^2$ as $\varepsilon\to 0$, we obtain that 
		$\phi^0_{ij}$ solves \eqref{eq:limit_linear_a} with $i=1$ in sense of distribution on  $\R^2$ and  $\phi^0_{ij}$ solves \eqref{eq:limit_linear_a} with $i\in \{ 2, \dots, N\}$ in sense of distribution  on $\R^2\setminus\{ 0\}$. 
		According to the regularity theory,
		$\phi^0_{ij}$ is a smooth solution on the whole space $\R^2$ for any $i=1,\dots, N, j=1,\dots,m.$
		The result in  \cite{DelPino2012Nondegeneracy} implies that the solution space of \eqref{eq:limit_linear_a} is generated by the following functions (in polar coordinate $(r,\theta)$)
		\[ \phi^0(r,\theta):=\frac{1-|r|^{\alpha_i}}{1+|r|^{\alpha_i}}, \,  \phi^1(r,\theta):=\frac{|r|^{\frac{\alpha_i}{2}}}{1+|r|^{\alpha_i}}\cos \frac{\alpha_i}{2} \theta, \,  \phi^2(r,\theta):=\frac{|r|^{\frac{\alpha_i}{2}}}{1+|r|^{\alpha_i}} \sin \frac{\alpha_i}{2} \theta.
		\]
		Observe that  $\phi^0$ is radial and for any $c_1, c_2\in \R$, 
		\[ c_1\phi^1(r, \theta)+c_2\phi^2(r,\theta)= \sqrt{
			c_1^2+c_2^2} \frac{|r|^{\frac{\alpha_i}{2}}}{1+|r|^{\alpha_i}}\sin \Big( \frac{\alpha_i}{2}\theta+ \theta_0\Big),\]
		where $\theta_0\in \R. $
		Since $\bphi\in \cH_{k}$ with $k> \frac{\alpha_N}{2}$, we have $c_1=c_2=0.$ It follows that 
		$ \tilde{\phi}_{ij} ^0 =c_{ij} \frac{1-|y|^{\alpha_i}}{1+|y|^{\alpha_i}}$ for some $c_{ij}\in\R$.\par 
		\begin{itemize}
			\item[{\bf Step 2.}]\label{item:step2_a}  	{\it For $i=1, \dots, N ,j=1,\dots
				,m, $    $\int_{\Omega_{ij}} 2\alpha_i^2 \frac{ |y|^{\alpha_i-2}}{(1+|y|^{\alpha_i})^2}\tilde{\phi}_{ij}(y) dy =o(|\log\varepsilon|^{-1}).$ }
		\end{itemize}
		Applying that $PZ_{ij}$ as a test function of~\eqref{eq:linear_key_a}, we have
		\begin{align*}
			\la \phi_i, PZ_{ij}\ra &=\int_{\Sigma} 2 \alpha_i^2 \chi_je^{-\varphi_j}  \frac{\delta_{i,j}^{\alpha_i}|y_{\xi_j}|^{\alpha_i-2}}{\left(\delta_{i,j}^{\alpha_i}+|y_{\xi_j}|^{\alpha_i}\right)^2} \phi_i PZ_{ij} dv_g \\
			&\quad+ \sum_{\substack{i'=1 \\i' \neq i}}^N \frac{a_{ii'}}2\int_{\Sigma} 2 \alpha_{i'}^2 \chi_je^{-\varphi_j}  \frac{\delta_{i',j}^{\alpha_{i'}}|y_{\xi_j}|^{\alpha_{i'}-2}}{\left(\delta_{i',j}^{\alpha_{i'}}+|y_{\xi_{i'}}|^{\alpha_{i'}}\right)^2} \phi_{i'}PZ_{ij} dv_g +\int_{\Sigma}h_i PZ_{ij}dv_g. 
		\end{align*}
		It follows that 
		\begin{align}\label{eq:lem3.1-2}
			0&=\int_{\Sigma} 2 \alpha_i^2 \chi_je^{-\varphi_j}  \frac{\delta_{i,j}^{\alpha_i}|y_{\xi_j}|^{\alpha_i-2}}{\left(\delta_{i,j}^{\alpha_i}+|y_{\xi_j}|^{\alpha_i}\right)^2} \phi_i\left( PZ_{ij}-Z_{ij}\right)  dv_g \\
			&\quad+ \sum_{\substack{i'=1 \\i' \neq i}}^N \frac{a_{ii'}}2\int_{\Sigma} 2 \alpha_{i'}^2 \chi_je^{-\varphi_j}  \frac{\delta_{i',j}^{\alpha_{i'}}|y_{\xi_j}|^{\alpha_{i'}-2}}{\left(\delta_{i',j}^{\alpha_{i'}}+|y_{\xi_{i'}}|^{\alpha_{i'}}\right)^2} \phi_{i'}PZ_{ij} dv_g +\int_{\Sigma}h_i PZ_{ij}dv_g.\nonumber
		\end{align}
		By Lemma~\ref{lem:extension_PZ_as},  we have $\|PZ_{ij}\|_{p'}=\mathcal{O}(1)$  for $\frac 1 p+\frac 1 {p'}=1$. Moreover, considering $\|\bh\|_p = o(|\log \varepsilon|^{-1})$, the H\"{o}lder inequality together with the Moser–Trudinger inequality yields
		\begin{equation}
			\label{eq:lem3.1-3} \left|\int_{\Sigma}h_i PZ_{ij}dv_g\right|\leq \|PZ_{ij}\|_{p'}\|h_i\|_{p}=\cO( \|h\|)=o(|\log \varepsilon|^{-1}).
		\end{equation}
		Applying Lemma~\ref{lem:extension_PZ_as} again, we derive that 
		\begin{align}\label{eq:lem3.1-4}
			&	(\log \varepsilon)\int_{\Sigma} 2 \alpha_i^2 \chi_je^{-\varphi_j}  \frac{\delta_{i,j}^{\alpha_i}|y_{\xi_j}|^{\alpha_i-2}}{\left(\delta_{i,j}^{\alpha_i}+|y_{\xi_j}|^{\alpha_i}\right)^2} \phi_i\left( PZ_{ij}-Z_{ij}\right)  dv_g\\ &= (\log \varepsilon)
			\int_{\Sigma} 2 \alpha_i^2 \chi_je^{-\varphi_j}  \frac{\delta_{i,j}^{\alpha_i}|y_{\xi_j}|^{\alpha_i-2}}{\left(\delta_{i,j}^{\alpha_i}+|y_{\xi_j}|^{\alpha_i}\right)^2} \phi_i\left(1 +\mathcal{O}(\delta_{i,j}^{2}|\log \delta_{i,j}|) \right)  dv_g\nonumber\\
			&=(1+\mathcal{O}(\delta_{i,j}^{2}|\log \delta_{i,j}|))(\log\varepsilon)\int_{\Omega_{ij}} 2\alpha_i^2 \frac{|y|^{\alpha_i-2}}{\left(1+|y|^{\alpha_i}\right)^2} \tilde{\phi}_{ij}(y)  dy+\mathcal{O}(\delta_{i,j}^{\alpha_i}|\log \varepsilon|)\nonumber\\
			&\stackrel{\eqref{def:delta_i}}{=}  (\log\varepsilon)\int_{\Omega_{ij}} 2\alpha_i^2 \frac{|y|^{\alpha_i-2}}{\left(1+|y|^{\alpha_i}\right)^2} \tilde{\phi}_{ij}(y)  dy+o(1), 
			\nonumber
		\end{align}
		as $ \varepsilon\rightarrow 0.$
		For $i\neq i'$, 
		By Lemma \ref{lem:extension_PZ_as}, we have  
		\begin{equation}
			\label{eq:extension_PZ_div}
			PZ_{ij}\circ y_{\xi_j}^{-1}(\delta_{i',j} y)=
			\left\{ \begin{array}{ll}
				\mathcal{O}\left( \left(\frac{\delta_{i,j}}{\delta_{i',j}}\right)^{\alpha_i} \frac 1 {|y|^{\alpha_i}}+\delta_{i,j}^{2}|\log\delta_{i,j}|\right)& \text{ if } i<i'\\
				2+\mathcal{O}\left(\left(\frac{\delta_{i',j}}{\delta_{i,j}}\right)^{\alpha_i} |y|^{\alpha_i}+\delta_{i,j}^{2}|\log\delta_{i,j}|\right)& \text{ if } i'<i
			\end{array}		\right..
		\end{equation}
		It follows that for $i\neq i'$
		\begin{align}\label{eq:lem3.1-5}
			&\quad 	(\log \varepsilon)\int_{\Sigma} 2 \alpha_{i'}^2 \chi_je^{-\varphi_j}  \frac{\delta_{i',j}^{\alpha_{i'}}|y_{\xi_j}|^{\alpha_{i'}-2}}{\left(\delta_{i',j}^{\alpha_{i'}}+|y_{\xi_j}|^{\alpha_{i'}}\right)^2} \phi_{i'} PZ_{ij}  dv_g\\ 
			&= (\log \varepsilon)
			\int_{\Omega_{i'j}}2 \alpha_{i'}^2  \frac{|y|^{\alpha_{i'}-2}}{\left(1+|y|^{\alpha_{i'}}\right)^2} \tilde{\phi}_{i'j}(y) PZ_{ij}\circ y_{\xi_j}^{-1}(\delta_{i',j} y)  dy+ \mathcal{O}(|\log\varepsilon|\delta_{i',j}^{\alpha_{i'}}) \nonumber\\
			&\stackrel{\eqref{eq:extension_PZ_div}}{=} \left\{\begin{array}{ll}
				\mathcal{O}\left( |\log\varepsilon|\int_{\Omega_{i'j}}2 \alpha_{i'}^2  \frac{|y|^{\alpha_{i'}-2}}{\left(1+|y|^{\alpha_{i'}}\right)^2} |\tilde{\phi}_{i'j}(y)| \left( \left(\frac{\delta_{i,j}}{\delta_{i',j}}\right)^{\alpha_i} \frac 1 {|y|^{\alpha_i}}+\delta_{i,j}^{2}\log \delta_{i,j}\right)  dy \right) &\\
				\hfill \text{
					if } i<i'	&\\
				2(\log\varepsilon) \int_{\Omega_{i'j}} 2\alpha_{i'}^2 \frac{|y|^{\alpha_{i'}-2}}{\left(1+|y|^{\alpha_{i'}}\right)^2} \tilde{\phi}_{i'j}(y) dy&\\
				+\mathcal{O}\left( |\log\varepsilon|\int_{\Omega_{i'j}}2 \alpha_{i'}^2  \frac{|y|^{\alpha_{i'}-2}}{\left(1+|y|^{\alpha_{i'}}\right)^2} |\tilde{\phi}_{i'j}(y)| \left( \left(\frac{\delta_{i',j}}{\delta_{i,j}}\right)^{\alpha_i}  |y|^{\alpha_i}+\delta_{i,j}^{\alpha_i}\right)  dy \right)&\\
				\hfill \text{ if } i'<i   & 
			\end{array}\right. \nonumber\\
			&= \left\{\begin{array}{ll}
				o(1) & \text{ if } i<i\\
				2(\log\varepsilon) \int_{\Omega_{i'j}} 2\alpha_{i'}^2 \frac{|y|^{\alpha_{i'}-2}}{\left(1+|y|^{\alpha_{i'}}\right)^2} \tilde{\phi}_{i'j}(y) dy +o(1) & \text{ if } i'<i
			\end{array}\right.. \nonumber
		\end{align}
		Substituting \eqref{eq:lem3.1-3}, \eqref{eq:lem3.1-4} and \eqref{eq:lem3.1-5} into \eqref{eq:lem3.1-2}, and applying induction on $
		i=1,\dots,N,$ we obtain 
		$(\log\varepsilon) \int_{\Omega_{ij}} 2\alpha_{i}^2 \frac{|y|^{\alpha_{i}-2}}{\left(1+|y|^{\alpha_{i}}\right)^2} \tilde{\phi}_{ij}(y) dy=o(1)$ as $\varepsilon\rightarrow 0$ for $i=1, \dots, N $.
		\begin{itemize}
			\item [{\bf Step 3.}]\label{item:step3_a}	{\it  $c_{ij}=0$ for $i=1,\cdots, N$ and $j=1,\cdots, m$. }
		\end{itemize} 
		Using $PU^i_j$ as a test function for~\eqref{eq:linear_key_a}, we derive that 
		\begin{align}\label{eq:test_PU^i_j}
			&\quad \int_{\Sigma} 2 \alpha_i^2 \chi_je^{-\varphi_j}  \frac{\delta_{i,j}^{\alpha_i}|y_{\xi_j}|^{\alpha_i-2}}{\left(\delta_{i,j}^{\alpha_i}+|y_{\xi_j}|^{\alpha_i}\right)^2} \phi_i dv_g\\
			&=\int_{\Sigma} 2 \alpha_i^2 \chi_je^{-\varphi_j}  \frac{\delta_{i,j}^{\alpha_i}|y_{\xi_j}|^{\alpha_i-2}}{\left(\delta_{i,j}^{\alpha_i}+|y_{\xi_j}|^{\alpha_i}\right)^2} \phi_i PU^i_j dv_g\nonumber\\
			&\quad+ \sum_{\substack{i'=1 \\i' \neq i}}^N \frac{a_{ii'}}2\int_{\Sigma} 2 \alpha_{i'}^2 \chi_je^{-\varphi_j}  \frac{\delta_{i',j}^{\alpha_{i'}}|y_{\xi_j}|^{\alpha_{i'}-2}}{\left(\delta_{i',j}^{\alpha_{i'}}+|y_{\xi_{i'}}|^{\alpha_{i'}}\right)^2} \phi_{i'}PU^i_j dv_g+\int_{\Sigma} h_i PU^i_j dv_g. \nonumber
		\end{align}
		The  $L.H.S$   of~\eqref{eq:test_PU^i_j} implies that 
		\begin{align*}
			\int_{\Sigma} 2 \alpha_i^2 \chi_je^{-\varphi_j}  \frac{\delta_{i,j}^{\alpha_i}|y_{\xi_j}|^{\alpha_i-2}}{\left(\delta_{i,j}^{\alpha_i}+|y_{\xi_j}|^{\alpha_i}\right)^2} \phi_i dv_g&=\int_{\Omega_{ij}} 2\alpha_i^2 \frac{ |y|^{\alpha_i-2}}{(1+|y|^{\alpha_i})^2}\tilde{\phi}_{ij}(y) dy+\mathcal{O}(\delta_{i,j}^{\alpha_i}) \\
			&=o(1),
		\end{align*}
		as established in \hyperref[item:step2_a]{Step 2}. 
		Moreover, by Lemma~\ref{lem:extension_PU_as}, it holds that  $\|PU^i_j\|_{\infty}= \mathcal{O}(|\log\varepsilon|)$. It follows that 
		$
		\left|\int_{\Sigma} h_i PU^i_j dv_g \right| \leq \|h_i\|_{p} \|PU^i_j\|_{p'}
		=o(1),
		$
		by the H\"{o}lder's inequality, 
		where $\frac 1 p+\frac{1}{p'}=1.$
		Then, by  H\"{o}lder's inequality together with the Moser–Trudinger inequality, we obtain the following estimate:
		\begin{equation*}
			\begin{split}
				&\delta_{i,j}\int_{\Omega_{ij}} 2 \alpha_i^2 \frac{|y|^{\alpha_i-2}}{\left(1+|y|^{\alpha_i} \right)^2}|\tilde{\phi}_{ij}(y)| |y| dy \\&\leq  2\alpha_i^2 
				\delta_{i,j} \delta_{i,j}^{\frac{2(1-q)}{q}}\left\|\phi_i\right\|_{q'}\left(\int_{\R^2}\left(\frac{|y|^{\alpha_j-1}}{\left(1+|y|^{\alpha_j}\right)^2}\right)^q d y\right)^{1 / q} =O\left(\delta_{i,j}^{\frac{2-q}{q}}\right)=o(1),
			\end{split}
		\end{equation*}
		where $q\in(1,2)$ such that $\frac 1 q +\frac 1 {q'}=1$.

		For $i=1, \dots, N ,$ applying Lemma~\ref{lem:extension_PU_as}, \hyperref[item:step1_a]{Step 1} and \hyperref[item:step2_a]{Step 2},  we deduce that 
		\begin{equation*}
			\begin{split}
				&\quad\int_{\Sigma} 2 \alpha_i^2 \chi_je^{-\varphi_j}  \frac{\delta_{i,j}^{\alpha_i}|y_{\xi_j}|^{\alpha_i-2}}{\left(\delta_{i,j}^{\alpha_i}+|y_{\xi_j}|^{\alpha_i}\right)^2} \phi_i PU^i_j dv_g\\ &=\int_{\Sigma} 2 \alpha_i^2  \frac{\delta_{i,j}^{\alpha_i}|y_{\xi_j}|^{\alpha_i-2}}{\left(\delta_{i,j}^{\alpha_i}+|y_{\xi_j}|^{\alpha_i}\right)^2} \phi_i \left(-2\chi_j\log\left( \delta_{i,j}^{\alpha_i}+ |y_{\xi_j}|^{\alpha_i} \right) +\frac {\alpha_i\varrho(\xi_j)} 2 H^g(\cdot,\xi_j)\right. \\
				&\left. \quad+ \mathcal{O}(\delta_{i,j}^{2}|\log \delta_{i,j}|)\right)  dv_g \\
				&=
				\int_{\Omega_{ij}}  \frac{2 \alpha_i^2|y|^{\alpha_i-2}}{\left(1+|y|^{\alpha_i}\right)^2} \tilde{\phi}_{ij}(y)\left(-2 \alpha_i \log \delta_{i,j}-2 \log \left(1+|y|^{\alpha_i}\right)+\frac{\alpha_i \varrho(\xi_j)}{2}R^g(\xi_j)\right) d y\\
				&\quad+O\left(\int_{\Omega_{ij}} 2 \alpha_i^2 \frac{|y|^{\alpha_i-2}}{\left(1+|y|^{\alpha_i} \right)^2}\left|\tilde{\phi}_{ij}(y)\right|\left(\delta_{i,j}|y|+\delta_{i,j}^{2}|\log\delta_{i,j}|\right) d y\right)+\mathcal{O}(\delta_{i,j}^{\alpha_i})\\
				&\rightarrow  -2 c_{ij} \int_{\R^2} 2 \alpha_i^2 \frac{|y|^{\alpha_i-2}}{\left(1+|y|^{\alpha_i}\right)^2}\frac{1- |y|^{\alpha_i}}{1+|y|^{\alpha_i}}\log(1+|y|^{\alpha_i}) dy= 4\alpha_i\pi c_{ij} ,
			\end{split}
		\end{equation*}
		as $ \varepsilon\rightarrow 0,$
		in which the last convergence used the facts that  $$\int_{\R^2} 2 \alpha_i^2 \frac{|y|^{\alpha_i-2}}{\left(1+|y|^{\alpha_i}\right)^2}\frac{1- |y|^{\alpha_i}}{1+|y|^{\alpha_i}} dy=0 $$
		and 
		$\int_{\R^2}  2\alpha_i^2 \frac{|y|^{\alpha_i-2}}{\left(1+|y|^{\alpha_i}\right)^2}\frac{1- |y|^{\alpha_i}}{1+|y|^{\alpha_i}}\log(1+|y|^{\alpha_i}) dy=-2\pi \alpha_i.$
		For any $i'\neq i$, by Lemma~\ref{lem:extension_PU_as}, \hyperref[item:step1_a]{Step 1} and \hyperref[item:step2_a]{Step 2} again, it follows that as $\varepsilon\rightarrow 0,$
		\begin{align*}
			&\quad\int_{\Sigma} 2 \alpha_{i'}^2 \chi_je^{-\varphi_j}  \frac{\delta_{i',j}^{\alpha_{i'}}|y_{\xi_j}|^{\alpha_{i'}-2}}{\left(\delta_{i',j}^{\alpha_{i'}}+|y_{\xi_j}|^{\alpha_{i'}}\right)^2} \phi_{i'} PU^i_j dv_g\\ &=\int_{\Sigma} 2 \alpha_{i'}^2 \chi_je^{-\varphi_j}  \frac{\delta_{i',j}^{\alpha_{i'}}|y_{\xi_j}|^{\alpha_{i'}-2}}{\left(\delta_{i',j}^{\alpha_{i'}}+|y_{\xi_j}|^{\alpha_{i'}}\right)^2} \phi_{i'}  \left(-2\chi_j\log\left( \delta_{i,j}^{\alpha_i}+ |y_{\xi_j}|^{\alpha_i} \right) \right. \\
			&\left. \quad+\frac {\alpha_i\varrho(\xi_j)} 2 H^g(\cdot,\xi_j)+ \mathcal{O}(\delta_{i,j}^{2}|\log \delta_{i,j}|)\right)  dv_g \\
			&	=\int_{\Omega_{i'j}} 2 \alpha_{i'}^2 \frac{|y|^{\alpha_{i'}-2}}{\left(1+|y|^{\alpha_{i'}}\right)^2} \tilde{\phi}_{i'j}\left(-2 \log \left(\delta_{i,j}^{\alpha_i}+\delta_{i',j}^{\alpha_i}|y|^{\alpha_i}\right)+\frac{\alpha_i \varrho(\xi_j)}{2}R^g(\xi_j)\right) d y \\
			&\quad+O\left(\int_{\Omega_{i'j}} 2 \alpha_{i'}^2 \frac{|y|^{\alpha_{i'}-2}}{\left(1+|y|^{\alpha_{i'}} \right)^2}\left|\tilde{\phi}_{i'j}(y)\right|\left(\delta_{i',j}|y|+\delta_{i,j}^{2}|\log\delta_{i,j}|\right) d y\right)+\mathcal{O}(\delta_{i',j}^{\alpha_{i'}})\\
			&\rightarrow  \left\{\begin{array}{ll}
				-2\alpha_i c_{i'j} \int_{\R^2} 2 \alpha_i^2 \frac{|y|^{\alpha_i-2}}{\left(1+|y|^{\alpha_i}\right)^2}\frac{1- |y|^{\alpha_i}}{1+|y|^{\alpha_i}}\log|y| dy & \text{ if } i<i'\\
				0 & \text{ if } i>i'
			\end{array} \right. \\
			&
			=\left\{\begin{array}{ll}
				8\pi \alpha_i  c_{i'j}	& \text{ if } i<i'\\
				0 & \text{ if } i>i'
			\end{array} \right.,
		\end{align*}
		in which  we used $\int_{\R^2} 2 \alpha_i^2 \frac{|y|^{\alpha_i-2}}{\left(1+|y|^{\alpha_i}\right)^2}\frac{1- |y|^{\alpha_i}}{1+|y|^{\alpha_i}} dy=0 $ and  $$\int_{\R^2} 2 \alpha_i^2 \frac{|y|^{\alpha_i-2}}{\left(1+|y|^{\alpha_i}\right)^2}\frac{1- |y|^{\alpha_i}}{1+|y|^{\alpha_i}}\log|y| dy=-4\pi.$$
		By summing up over the index $i'\in \{1,\dots, N\}$,  
		the  $R.H.S.$   of~\eqref{eq:test_PU^i_j} has the following estimate:
		\begin{equation*}
			\left\{\begin{array}{ll}
				4\pi \alpha_i	\left( c_{ij}-  \sum_{i'>i}(-a_{ii'}) c_{i'j}\right) +o(1) &\text{ if } i=1, \cdots, N-1\\
				4\pi \alpha_i c_{ij}+o(1) & \text{ if } i=N
			\end{array}\right..
		\end{equation*}
		Consequently,  $c_{ij}=0$ for any $i=1,\dots, N$, and $j=1,\dots,m.$
		
		\begin{itemize}
			\item [{\bf Step 4.}]\label{item:step4_a}	{\it  Construct a contradiction. }
		\end{itemize} 
		We  reduce the Cartan matrix $(a_{ij})$ to be diagonal using Gaussian elimination as follows:  
		\begin{equation}\label{eq:lem3.1-10}
			-\Delta_g	\left(\begin{matrix}
				\frac 4 3 \phi_1+ \frac  2 3 \phi_ 2 \\
				\frac 1 2 \phi_1+ \frac 3 2 \phi_2+ \frac 3 4 \phi_3 \\
				\vdots\\
				\frac{
					N-3}{N-2}\phi_{N-3}+ (1+\frac{N-2}{N})\phi_{N-2}+ \frac{N-1}{N} \phi_{N-1}
				\\
				\frac{N-2}{N-1}\phi_{N-2}+ ( 1+
				\frac{a_*}{2-a_*})\phi_{N-1} + \frac{ -a_{N-1 N}}{2-a_*} \phi_N \\
				\phi_N + (-a_{N N-1})\frac{N-1}{N}\phi_{N-1}
			\end{matrix}\right)= \frac 12  (\tilde{a}_{ij})  \mathbf{\Psi}+o(1) ,
		\end{equation}
		where $a_* = \frac{N-1}{N}(a_{NN-1} a_{N-1 N})=\begin{cases}
			\frac {N-1} N & \text{ for }\mathbf{A}_N \\
			2 \frac{N-1}{N} & \text{ for } \mathbf{B}_N,\mathbf{C}_N\\
			\frac 3 2    & \text{ for }\mathbf{G}_2
		\end{cases}, $ 
		$$(\tilde{a}_{ij}):= \left(
		\begin{matrix}
			2 & 0 &  \cdots&0 & 0 \\
			0 & \frac 3 2 &  \cdots&0 & 0 \\
			\vdots & \vdots & \ddots& \vdots &\vdots \\
			0 & 0 &  \cdots &\frac{N}{N-1} &  0\\
			0 & 0 &  \cdots & 0  &2-a_*\\
		\end{matrix}
		\right)$$ 
		and 
		\[ \mathbf{\Psi}=  \left(\begin{matrix}
			\sum_{j=1}^m\left( \chi_j e^{-\varphi_j} |y_{\xi_j}|^{\alpha_1-2} e^{U^1_j} \phi_1-\overline{\chi_j e^{-\varphi_j} |y_{\xi_j}|^{\alpha_1-2} e^{U^1_j} \phi_1}\right)\\
			\vdots\\
			\sum_{j=1}^m\left( \chi_j e^{-\varphi_j} |y_{\xi_j}|^{\alpha_N-2} e^{U^N_j} \phi_N-\overline{\chi_j e^{-\varphi_j} |y_{\xi_j}|^{\alpha_N-2} e^{U^N_j} \phi_N}\right)
		\end{matrix}\right).\]
		Since $\tilde{\phi}_{ij} \rightarrow 0$ in $\rL^{\alpha_i}$, we have
		\[  \sum_{j=1}^m\frac 1 2 \int_{\Sigma} \chi_je^{-\varphi_j} |y_{\xi_j}|^{\alpha_{i}-2} e^{U^{i}_j} \phi_{i}^2  dv_g=  \sum_{j=1}^m \cO(\|\tilde{\phi}_{ij}\|^2_{\rL^{\alpha_i}})  =o(1).\]
		Using $\phi_i$ as a test function of \eqref{eq:lem3.1-10} for $i=1,\dots, N$,  we deduce that 
		\begin{equation*}
			\begin{split}
				&	(1+\frac 1 3 )\|\phi_1\|^2 + \frac 2 3 \la \phi_2, \phi_1\ra 	= o(1),\\
				& (1+ \frac  12 )\|\phi_2\|^2+ \frac 1 2 \la \phi_1, \phi_2\ra + \frac  3 4 \la \phi_3, \phi_2 \ra  =
				+o(1),\\
				&\cdots \cdots\\
				& (1+\frac{N-2}{N})\|\phi_{N-2}\|^2 + \frac{N-3}{N-2} \la \phi_{N-3}, \phi_{N-2}\ra +\frac{N-1}{N} \la \phi_{N-1}, \phi_{N-2} \ra =o(1),\\
				&( 1+ \frac{a_*}{2-a^*})	\|\phi_{N-1}\|^2 +\frac {N-2}{N-1} \la \phi_{N-2},\phi_{N-1}\ra+\frac {-a_{N-1 N} } {2-a_*}\la \phi_{N},\phi_{N-1} \ra=o(1),\\
				& \|\phi_N\|^2+\frac{N-1}{N}( -a_{NN-1}) \la \phi_{N-1}, \phi_N \ra =o(1), 
			\end{split}
		\end{equation*}
		where we applied \eqref{eq:h_phi},  and $\int_\Sigma \phi_i dv_g=0.$ 
		By summing up the estimates above, the Cauchy–Schwarz inequality and  Young's inequality yield that 
		\begin{equation*}
			\frac 1 2\Big(\frac 1 N -\frac 1 {N+1}\Big)	\sum_{i=1}^N \|\phi_i\|^2= o(1), 
		\end{equation*}
		which
		contradicts with the assumption $\|\bphi\|=1$. 
		Lemma~\ref{lem:invertible_as} is concluded. 
	\end{proof}
	
	For any $p>1$, let $i^*_p: L^p(\Sigma)\rightarrow \oH$  be the adjoint operator corresponding to the immersion $i: \oH \rightarrow L^{\frac{p}{p-1}}$ and $\tilde{i}^* : \cup_{p>1} L^p(\Sigma)\rightarrow \oH$. For any $f\in L^p(\Sigma)$, we define that 
	$i^*(f):=\tilde{i}^*( f-\bar{f})$, i.e. for any $h\in \oH $, 
	$\la i^*(f), h\ra =\int_{\Sigma}(f-\bar{f}) h dv_g.$
	From  \cite{yang2006extremal}, we have the Moser-Trudinger type inequality  on compact Riemann surfaces as follows: 
	\begin{equation*} \begin{array}{lll}
			\log \int_{\Sigma} e^u dv_g &\leq&
			\frac 1 {8\pi} \int_{\Sigma} |\nabla_g u|^2 dv_g  +C, \text{ for any } u\in \oH, 
	\end{array}\end{equation*}
	where $C>0$ is a constant. Then it follows  $\oH\rightarrow L^p(\Sigma), u\mapsto e^u$ is continuous.
	
	Let 
	$ f_i(u_i)= 2\varepsilon V_i e^{u_i}-\overline{2\varepsilon V_i e^{u_i}}, \text{ for } i=1,\dots, N, 
	$
	and 
	\begin{equation}
		F(\bu)=\left(\begin{array}{c}
			f_1(u_1)+\sum_{i'\neq 1}\frac {a_{1i'}}{2} f_{i'}(u_{i'})\\
			\vdots\\
			f_N(u_N)+ \sum_{i'\neq N} \frac {a_{Ni'}}{2} f_{i'}(u_{i'})
		\end{array}\right). 
	\end{equation}
	The problem~\eqref{eq:todaN} has the following equivalent form:
	\begin{equation} \label{eq:todaN2}
		\left\{\begin{array}{ll}
			\bu =(u_1,u_2)=i^*(F(\bu)) \\
			u\in \cH
		\end{array}\right.. 
	\end{equation}
	We define that
	$ \cL(\bphi)=i^*\left(  \cL_{\xi,\varepsilon}(\bphi)\right),$
	for $\bphi=(\phi_1,\dots, \phi_{N})\in \cH_{k}$. 
	\begin{proposition}
		There exist $\varepsilon_0>0$ and $C>0$ such that for any $\bh\in \cH_{k}$, there exist a constant $C>0$ and a unique solution $\bphi\in\cH_{k}$ solving 
		\begin{equation}\label{eq:L_bphi}
			\cL(\bphi)=\bh,
		\end{equation}
		with 
		$ \|\cL^{-1}\|\leq C |\log \varepsilon|,$
		for any $\varepsilon\in (0,\varepsilon_0).$
	\end{proposition}
	\begin{proof}
		We observe that $\bphi\mapsto i^*\left(\chi_j e^{-\varphi_j} |y_{\xi_j}|^{\alpha_i-2} e^{U^i_j} \phi_i\right)$ is a compact operator in $\cH_{k}$ for any $j=1,\dots,m, i=1, \dots, N $.
		Consequently, we infer that $\cL$ is  a Fredholm operator in $\cH_{k}.$
		Suppose that $\bphi$ solves $\cL(\bphi)=0.$ By Lemma~\ref{lem:invertible_as}, we deduce that $\bphi=0$. From this deduction and the application of Fredholm's alternative, we can establish the uniqueness and existence of solutions for~\eqref{eq:L_bphi} in $\cH_{k}$. 
		
		It remains to get the estimate of the inverse operator of $\cL$. Given any $\bh\in \cH_{k}$, we already proved that there exists a unique $\bphi=(\phi_1,\cdots, \phi_N)\in \cH_{k}$ such that 
		$i^*\circ\cL_{\xi,\varepsilon}(\bphi)=\bh$. Lemma~\ref{lem:invertible_as} implies that 
		\begin{equation*}
			\|\cL^{-1}(\bh)\|=	\|\bphi\|\leq C |\log\varepsilon|\|\cL_{\xi,\varepsilon}(\bphi)\|_p\leq  C |\log\varepsilon|
			\|\cL(\bphi)\|\leq C|\log\varepsilon| \|\bh\|,
		\end{equation*}
		where $p>1$.
	\end{proof}
	\section{Nonlinear problem}\label{sec:4}
	
	In this section, firstly, we show the error terms associated with the approximate solutions is sufficiently small. We then establish the solvability of the nonlinear problem using the contraction mapping principle.
	
	The expected solution $\bW_{\varepsilon}+\bphi_{\varepsilon}$ solves~\eqref{eq:todaN} if and only if 
	$\bphi_{\varepsilon}$ solves the following problem in $\cH_{k}$:
	\begin{equation}
		\label{eq:toda_as_nonlinear}
		\cL_{\xi,\varepsilon}(\bphi) =\cS_{\xi,\varepsilon}(\bphi)+ \cN_{\xi,\varepsilon}(\phi)+\cR_{\xi,\varepsilon}.
	\end{equation}
	Here, $\cL_{\xi,\varepsilon}$ denotes the linear operator defined in~\eqref{eq:linear_op_a}, and $$\mathcal{S}_{\xi,\varepsilon}(\bphi) := \left(S_{\xi,\varepsilon}^1(\bphi), \dots, S_{\xi,\varepsilon}^N(\bphi)\right),$$ denotes a higher-order linear operator, defined    for each $i = 1, \dots, N$ by
	\begin{equation*}
		\begin{split}
			S_{\xi,\varepsilon}^i(\bphi)&:= \Big(-\sum_{j=1}^m\chi_j e^{-\varphi_j}|y_{\xi_j}|^{\alpha_i-2} e^{U^i_j}+2 \varepsilon V_ie^{ W_{i,\varepsilon}}\Big) \phi_i\\
			&\quad-\overline{\Big(-\sum_{j=1}^m\chi_j e^{-\varphi_j}|y_{\xi_j}|^{\alpha_i-2} e^{U^i_j}+2 \varepsilon V_ie^{ W_{i,\varepsilon}}\Big) \phi_i}\\
			&\quad+\sum_{\substack{i'=1\\ 
					i' \neq i}}^N(-a_{ii'}) \Big(\sum_{j=1}^m \frac{1}{2}\chi_j e^{-\varphi_j}|y_{\xi_j}|^{\alpha_{i'}-2} e^{U^{i'}_j}-\varepsilon V_{i'}e^{W_{i',\varepsilon}}\Big) \phi_{i'}\\
			&\quad-\overline{ \sum_{\substack{i'=1\\ 
						i' \neq i}}^N (-a_{ii'}) \Big(\sum_{j=1}^m \frac{1}{2}\chi_j e^{-\varphi_j}|y_{\xi_j}|^{\alpha_{i'}-2} e^{U^{i'}_j}-\varepsilon V_{i'} e^{W_{i',\varepsilon}}\Big) \phi_{i'} }, 
		\end{split}
	\end{equation*}
	$\mathcal{N}_{\xi,\varepsilon}(\bphi):=\left(N_{\xi,\varepsilon}^1(\bphi),\dots,  N_{\xi,\varepsilon}^N(\bphi)\right)$ is 	the nonlinear term, defined by
	\begin{equation*}
		\begin{split}
			N_{\xi,\varepsilon}^i(\bphi):&=2 \varepsilon V_i  e^{W_{i,\varepsilon}}\left(e^{\phi_i}-1-\phi_i\right)-\overline{2 \varepsilon V_i  e^{W_{i,\varepsilon}}\left(e^{\phi_i}-1-\phi_i\right)}\\
			&\quad -  \sum_{\substack{i'=1 \\ i' \neq i}}^N(-a_{ii'})\left( \varepsilon V_{i'} e^{W_{i',\varepsilon}}\left(e^{\phi_{i'}}-1-\phi_{i'}\right)-\overline{ \varepsilon V_{i'} e^{W_{i',\varepsilon}}\left(e^{\phi_{i'}}-1-\phi_{i'}\right) }\right), 
		\end{split}
	\end{equation*}
	for each $i = 1, \dots, N$, and  $\cR_{\xi,\varepsilon}:=(R^1_{\xi,\varepsilon},\dots, R^N_{\xi,\varepsilon} )$ is the error term, defined for each $i = 1, \dots, N$ by
	\begin{equation*}
		\begin{split}
			R^{i}_{\xi,\varepsilon}&=\Delta_g W_{i,\varepsilon}+2\varepsilon V_i e^{W_{i,\varepsilon}}- \sum_{\substack{i'=1 \\ i' \neq i}}^N (-a_{ii'})\varepsilon V_{i'} e^{W_{i',\varepsilon}}\\
			&\quad-\overline{\Big(2\varepsilon V_i e^{W_{i,\varepsilon}}- \sum_{\substack{i'=1 \\ i' \neq i}}^N (-a_{ii'})\varepsilon V_{i'} e^{W_{i',\varepsilon}}\Big)}. 
		\end{split}
	\end{equation*}
	Firstly, we will show that the error term goes to zero along $\varepsilon\rightarrow 0$.  
	To address this, we introduce a family of functions  to deal with the interactions among projected bubbles.
	For $i=1, \dots, N ,j=1,\dots,m$, and $y\in \Omega_{ij}:=\frac 1 {\delta_{i,j}} B^{\xi_j} $, 
	\begin{align}
		\label{eq:def_theta_as}
		\Theta_{ij}(y)&:=
		\exp\Big\{ \varphi_j\circ y_{\xi_j}^{-1}(\delta_{i,j}y)+ \Big(PU^i_j-U^i_j+\sum_{\sumi}^N \frac {a_{ii'}}2 PU^{i'}_j\Big)\circ y_{\xi_j}^{-1}(\delta_{i,j} y) \\
		&
		\quad	 +\sum_{j'\neq j} \Big( PU^i_{j'}+\sum_{\sumi}^N \frac {a_{ii'}}2 PU^{i'}_{j'}\Big)\circ y_{\xi_j}^{-1}(\delta_{i,j} y)+\log V_i\circ y_{\xi_j}^{-1}(\delta_{i,j} y)\nonumber\\
		&\quad+\log(a_{ii}\varepsilon)-(\alpha_i-2)\log|\delta_{i,j} y| \Big\}.\nonumber
	\end{align}
	We take $d_{i,j}$ with the value in~\eqref{eq:def_d_ij_as} to ensure that  $\Theta_{ij}$ is sufficiently small for $i=1, \dots, N , j=1,\dots,m$, as stated in Lemma~\ref{lem:theta_small_as}. 
	
	\begin{lemma}\label{lem:est_error_as}
		There exist $p_0>1$ and $\varepsilon_0>0$ such that for any $p \in\left(1, p_0\right)$ and $\varepsilon \in\left(0, \varepsilon_0\right)$ we have
		$	\left\|\mathcal{R}_{\xi,\varepsilon}\right\|_p:=\sum_{i=1}^N \|R^i_{\xi,\varepsilon}\|_{p}=O\left(\varepsilon^{\frac 1 {4N}\frac{2-p}{p }}\right).
		$
	\end{lemma}
	\begin{proof} By the definition, the $i$-th component of the error term can be expressed as follows: 
		\begin{equation*}
			\begin{split}
				R^i_{\xi,\varepsilon} 
				&=\Big(2\varepsilon V_i e^{W_{i,\varepsilon}}-\sum_{j=1}^m  \chi_j e^{-\varphi_j} |y_{\xi_j}|^{\alpha_i-2} e^{U^i_j}\Big)\\
				&\quad- \frac{1}{2} \sum_{\substack{i'=1 \\ i' \neq i}}^N (-a_{ii'})\Big( 2\varepsilon V_{i'} e^{W_{i',\varepsilon}}- \sum_{j=1}^m 
				\chi_j e^{-\varphi_j} |y_{\xi_j}|^{\alpha_{i'}-2} e^{U^{i'}_j} \Big)  \\
				&  \quad-\overline{ \Big(2\varepsilon V_i e^{W_{i,\varepsilon}}-\sum_{j=1}^m  \chi_j e^{-\varphi_j} |y_{\xi_j}|^{\alpha_i-2} e^{U^i_j}\Big)}  \\
				&\quad+ \overline{\frac 1 2\sum_{\substack{i'=1 \\ i' \neq i}}^N (-a_{ii'})\Big( 2\varepsilon V_{i'} e^{W_{i',\varepsilon}}- \sum_{j=1}^m 
					\chi_j e^{-\varphi_j} |y_{\xi_j}|^{\alpha_{i'}-2} e^{U^{i'}_j} \Big)}.
			\end{split}
		\end{equation*}
		For any $p>1$, the H\"{o}lder's inequality  together with  Lemma \ref{lem:diif_e^W_sum_e^u_as} yields that 
		\begin{align*}
			&\quad	\Big|\int_{\Sigma}\Big(\sum_{j=1}^m  \chi_j e^{-\varphi_j} |y_{\xi_j}|^{\alpha_i-2} e^{U^i_j}-2\varepsilon  V_i e^{W_{i,\varepsilon}} \Big)dv_g \Big|\\
			&\leq  \Big\|\sum_{j=1}^m  \chi_j e^{-\varphi_j} |y_{\xi_j}|^{\alpha_i-2} e^{U^i_j}-\varepsilon  V_i e^{W_{i,\varepsilon}}\Big\|_{p}\leq C \varepsilon^{\frac 1 {4N}\frac{2-p}{p}},
		\end{align*}
		for some constant $C>0$,	in view of $|\Sigma|_g =1.$
		Thus, we conclude~Lemma \ref{lem:est_error_as}. 
	\end{proof}
	The following lemma shows that the higher order linear operator $\cS_{\xi,\varepsilon}$ is bounded on $\cH$ and the operator norm vanishes as $\varepsilon\rightarrow 0$. 
	\begin{lemma}
		There exist $p_0 > 1$ and $\varepsilon_0 > 0$ such that for any $p \in (1, p_0)$ and $\varepsilon \in (0, \varepsilon_0)$,
		\[
		\| \cS_{\xi,\varepsilon}(\bphi) \|_p = O\left( \varepsilon^{\frac 1 {4N}\frac{2-p}{ p}} \| \bphi \|  \right), \text{ for }\bphi\in \cH.
		\]
	\end{lemma}
	\begin{proof}
		Applying Lemma~\ref{lem:diif_e^W_sum_e^u_as} along with H\"{o}lder's inequality and the Sobolev inequality, we obtain
		\begin{equation*}
			\begin{split}
				\| \cS_{\xi,\varepsilon}(\bphi) \|_p &= \sum_{i=1}^N \| S^{i}_{\xi,\varepsilon}(\bphi) \|_{p}
				\\
				&=O\Big( \sum_{i=1}^{N} \Big\| \Big(-\sum_{j=1}^m\chi_j e^{-\varphi_j}|y_{\xi_j}|^{\alpha_i-2} e^{U^i_j}+2 \varepsilon V_ie^{ W_{i,\varepsilon}}\Big) \phi_i\Big\|_{p}\Big)\\
				&\quad+ O\Big( \sum_{i=1}^N \Big| \int_{\Sigma} \Big(-\sum_{j=1}^m\chi_j e^{-\varphi_j}|y_{\xi_j}|^{\alpha_i-2} e^{U^i_j}+2 \varepsilon V_ie^{ W_{i,\varepsilon}} \Big)\phi_i  dv_g \Big|\Big)\\
				&= O\Big( \sum_{i=1}^N \Big( \Big\|2 \varepsilon V_ie^{W_{i,\varepsilon}}-\sum_{j=1}^m\chi_j e^{-\varphi_j}|y_{\xi_j}|^{\alpha_i-2} e^{U^i_j} \Big\|_{p} \Big) \| \phi_i\|_{\frac{p}{p-1}} \Big)  \\
				&= O\Big(\varepsilon^{\frac 1 {4N}\frac{2-p}{ p}} \sum_{i=1}^N \| \phi_i \| \Big) = O\Big(\varepsilon^{\frac{1}{4N}\frac{2-p}{ p}} \|\bphi\|\Big),
			\end{split}
		\end{equation*}
		as $\varepsilon\rightarrow 0$ for any $p>1$ sufficiently close to $1$. 
	\end{proof}

	To study the asymptotic behavior of  the non-linear part $\cN_{\xi,\varepsilon}$, we have the following lemma: 
	\begin{lemma}\label{lem:non_linear_term_as}
		There exist $s_0 > 1$ and $\varepsilon_0 > 0$ such that for any $p > 1$, $r > 1$ with $pr \in (1, s_0)$ and $\varepsilon \in (0, \varepsilon_0)$, 
		\[
		\|\mathcal{N}_{\xi,\varepsilon}(\bphi)\|_p = O\left( \varepsilon^{\frac 1 {2N}\frac{1-pr}{pr}} \|\bphi\|^2\right) 
		\]
		and
		\[
		\|\mathcal{N}_{\xi,\varepsilon}(\bphi^0) - \mathcal{N}_{\xi,\varepsilon}(\bphi^1)\|_p = O \left( \varepsilon^{\frac 1 {2N}\frac{1-pr}{pr}} \|\bphi^0 - \bphi^1\| (\|\bphi^1\| + \|\bphi^0\|) \right) 
		\]
		hold true for any   $\bphi, \bphi^0, \bphi^1 \in \{ \bphi=(\phi_1,\phi_2) \in \cH: \|\phi_i\| \leq 1, i=1, \dots, N \}$. 
	\end{lemma}
	\begin{proof}
		Analogous to the proof of Lemma~\ref{lem:non_linear_basic} we can obtain the result. The detail is omitted here. 
	\end{proof}
	Next, we will use the fixed point theorem to solve the non-linear problem~\eqref{eq:toda_as_nonlinear}.
	\begin{proposition}
		\label{prop:fixed_point}
		Given 
		$p_0>1$, $\varepsilon_0>0$ and $R_0>0$  such that 
		for any $ p \in (1, p_0) $, $ \varepsilon \in (0, \varepsilon_0) $ and $ R \geq R_0 $,  there exists a unique  $ \bphi_\varepsilon = (\phi^1_\varepsilon,\dots \phi^N_\varepsilon) \in \cH_{k} $ for  $\varepsilon\in (0,\varepsilon_0)$, such that  $\bW_{\varepsilon}+\bphi_{\varepsilon}$ solves the Toda system~\eqref{eq:todaN} with the parameter $\varepsilon$, satisfying 
		\[ \|\bphi_{\varepsilon}\|\leq R \varepsilon^{\frac 1 {4N}\frac {2-p} {p}}|\log \varepsilon|. \]
	\end{proposition}
	\begin{proof}
		We define the operator 	\[ \cT_{\xi,\varepsilon}(\bphi):= \cL_{\xi,\varepsilon}^{-1}(\cS_{\xi,\varepsilon}(\bphi)+\cN_{\xi,\varepsilon}(\bphi)+\cR_{\xi,\varepsilon}) \]
		on $\cH_{k}.$ To find $\bphi_{\varepsilon}\in \cH_{k}$  such that $\bW_{\varepsilon}+\bphi_{\varepsilon}$ solves~\eqref{eq:todaN}, it is sufficient to find a fixed point of $\cT_{\xi,\varepsilon}: \cH_{k}\rightarrow\cH_{k}$. 
		We choose $p_0, r>1$ sufficiently close to $1$ such that 
		\[ \frac{2-p_0}{4N p_0}+ \frac{1-p_0r}{2N p_0r}>0. \]
		By Lemma~\ref{lem:invertible_as} and Lemma~\ref{lem:est_error_as}-Lemma~\ref{lem:non_linear_term_as}, we deduce that there exist $R_0>1$ sufficiently large and $\varepsilon_1>0$ sufficiently small such that for any $\varepsilon\in (0,\varepsilon_1), R\geq R_0, p\in (1,p_0)$
		\begin{equation*}
			\begin{split}
				\|\cT_{\xi,\varepsilon}(\bphi)\|&\leq C|\log\varepsilon|\left( \|\cS_{\xi,\varepsilon}(\bphi)\|_p+\|\cN_{\xi,\varepsilon}(\bphi)\|_p+\|\cR_{\xi,\varepsilon}\|_p\right)\\
				&\leq C|\log \varepsilon|\left(  \varepsilon^{\frac 1 {4N}\frac{2-p}{p}} \| \bphi \|+ \varepsilon^{\frac 1 {2N}\frac{1-pr}{pr}} \|\bphi\|^2 + \varepsilon^{\frac 1 {4N}\frac{2-p}{p}} \right)\\
				&\leq  R \varepsilon^{\frac 1 {4N}\frac{2-p}{p}}|\log \varepsilon|,
			\end{split}
		\end{equation*}
		for any $\bphi\in \{\bphi\in\cH: \|\bphi\|\leq R \varepsilon^{\frac 1 {4N}\frac{2-p}{p}}|\log \varepsilon|\}. $
		Next, we will prove that $\cT_{\xi,\varepsilon}$ is a contract mapping. 
		Given $\bphi^0,\bphi^1\in \{\bphi\in\cH: \|\bphi\|\leq R \varepsilon^{\frac 1 {4N}\frac{2-p}{p}}|\log \varepsilon|\},$ Lemma~\ref{lem:invertible_as} and  Lemma~\ref{lem:est_error_as}-\ref{lem:non_linear_term_as} imply that there exists $\varepsilon_0\in(0,\varepsilon_1)$ such that for any $\varepsilon\in (0,\varepsilon_0)$
		\begin{equation*}
			\begin{split}
				&	\|\cT_{\xi,\varepsilon}(\bphi^1)-\cT_{\xi,\varepsilon}(\bphi^0)\|\leq C |\log\varepsilon|\left( \|S_{\xi,\varepsilon}(\bphi^0-\bphi^1)\|_p+\|\cN_{\xi,\varepsilon}(\bphi^0)-\cN_{\xi,\varepsilon}(\bphi^1)\|_p\right)\\
				&\leq C|\log\varepsilon|\left(\varepsilon^{\frac 1 {4 N} \frac{2-p}{p}} \| \bphi^0-\bphi^1 \| +\varepsilon^{\frac 1 {2N}\frac{1-pr}{ pr}} (\|\bphi^0\|+\|\bphi^1\|)\|\bphi^0-\bphi^1\| \right)\\
				&\leq C\left( \varepsilon^{\frac 1 {4 N}\frac{2-p}{p}}|\log\varepsilon|+ \varepsilon^{\frac 1 {4N}\frac{2-p}{p}+\frac 1 {2N} \frac{1-pr}{pr}}|\log\varepsilon|^2\right)\|\bphi^0-\bphi^1\|\leq \frac 1 2 \|\bphi^0-\bphi^1\|.
			\end{split}
		\end{equation*}
		Utilizing the contraction mapping principle, we establish that for any $\varepsilon\in (0,\varepsilon_0)$ there exists a fixed point $\bphi_{\varepsilon}\in \left\{ \bphi\in \cH_{k}: \|\bphi\|\leq R \varepsilon^{\frac 1 {4N}\frac{2-p}{p}}|\log\varepsilon|\right\}$ for the operator $\cT_{\xi,\varepsilon}.$ 
	\end{proof}
	\section{Proof of main results}\label{sec:5}
	This section completes the proof of Theorem \ref{thm:main_asymmetric} by combining the results from the previous sections. By Proposition \ref{prop:fixed_point}, to finalize the construction of the blow-up solutions, it suffices to verify their properties (ii) and (iii).
	
	\begin{proof}
		[Proof of Theorem \ref{thm:main_asymmetric}]
		We firstly prove ii), that is   $$ \rho^{\varepsilon}=(\rho^\varepsilon_1,\dots,\rho^\varepsilon_N)\rightarrow (2\pi m \alpha_1, \dots, 2\pi m \alpha_N)$$ as $\varepsilon\rightarrow 0.$
		Applying Lemma~\ref{lem:theta_small_as} and Proposition~\ref{prop:fixed_point}, we  deduce that 
		\begin{equation*}
			\begin{split}
				\rho^\varepsilon_i&:=\int_{\Sigma} \varepsilon V_i e^{u_{i,\varepsilon}} dv_g \\
				&= \sum_{j=1}^m \int_{U(\xi_j)}  \frac{ \alpha_i^2 |y_{\xi_j}|^{\alpha_i-2} }{(\delta_{i,j}^{\alpha_i}+|y_{\xi_j}(x)|^{\alpha_i})^2} \exp \Big\{ \Big(PU^i_j-U^i_j+ \frac {a_{ii'}}2\sum_{\sumi}^N PU^{i'}_j\Big)  \\
				& \quad+\sum^m_{\substack{j'=1
						\\
						j'\neq j}} \Big(PU^{i}_{j'}+\frac {a_{ii'}}2 \sum_{\sumi}^N PU^{i'}_{j'}\Big)+ \log V_i +\log(2\varepsilon)-(\alpha_i-2) \log |y_{\xi_j}|\\
				&\quad+ \phi_{1,\varepsilon}\Big\} dv_g +\mathcal{O}(\varepsilon)= \sum_{j=1}^m \int_{\Omega_{ij}}  \frac{\alpha_i^2|y|^{\alpha_i-2}}{(1+|y|^{\alpha_i})^2} e^{\Theta_{ij}} dy +o(1)\\
				&=\sum_{j=1}^m \int_{\Omega_{ij}}  \frac{\alpha_i^2|y|^{\alpha_i-2}}{(1+|y|^{\alpha_i})^2} (1+\cO(\delta_{i,j}|y|+\varepsilon^{\frac 1{2i}}) ) dy +o(1)\\
				&=\sum_{j=1}^m \frac{\alpha_i\varrho(\xi_j)}{4} +o(1)=\sum_{j=1}^m \frac{\alpha_i\varrho(\xi^*_j)}{4} +o(1)= 2\pi \alpha_i m +o(1),
		\end{split}	\end{equation*}
		where we applied $\xi^*_j\in \intsigma$,  the integrals 
		$\int_{\R^2} \frac {1}{(1+|y|^2)^2} dy=\pi$ and $\int_{\R^2} \frac{|y|^2}{(1+|y|^4)^2} dy =\frac \pi 2 . $ 	Based on the result in Proposition~\ref{prop:fixed_point},  Theorem~\ref{thm:main_asymmetric} is concluded. 
		
		To prove the property iii), we introduce an inequality $|e^s-1|\leq C e^{|s|}|s|$ for all $s\in \R$. For any $\psi\in C^\infty(\Sigma)$, by Proposition \ref{prop:fixed_point} and Lemma  \ref{lem:diif_e^W_sum_e^u_as},  we have
		\begin{equation*}
			\begin{split}
				\int_{\Sigma} \varepsilon V_i e^{u_{i,\varepsilon}}\psi dv_g &= 	\int_{\Sigma} \varepsilon V_i e^{W_{i,\varepsilon}}\psi dv_g  +o(1)\\
				&=\sum_{j=1}^m\frac 1 2 \int_{\Sigma}  \chi_j e^{-\varphi_j} |y_{\xi_j}|^{\alpha_i-2} e^{U^i_j}  \psi dv_g\\
				& \quad +\cO\Big(\int_\Sigma\frac 1 2 \Big|\sum_j \chi_j e^{-\varphi_j} |y_{\xi_j}|^{\alpha_i-2} e^{U^i_j}  -\varepsilon  V_i e^{W_{i,\varepsilon}}\Big|\cdot|\psi| dv_g\Big)+o(1)\\
				&=\sum_{j=1}^m\frac {\alpha_i  \varrho(\xi^*_j)} 4  \psi(\xi_j^*)+o(1)= \sum_{j=1}^m 2\pi \alpha_i \psi(\xi_j^*)+o(1). 
			\end{split}
		\end{equation*}
		Thus, iii) is concluded.
	\end{proof}
	
	%%%%
	\appendix
	\section{Technique estimates}\label{App:A}
	In Appendix \ref{App:A}, we provide detailed technical estimates that support the main analysis.
	
	The following lemma shows the asymptotic behavior of the projected bubbles $PU^i_j$ for $j=1,\dots, m$ and $i=1,\dots, N $. 
	\begin{lemma}\label{lem:extension_PU_as}
		For  $j=1,\dots, m$ and $i=1,\dots, N $,  it holds:
		$$PU^i_j= \chi_j\cdot ( U^i_j-\log( 2\alpha_i^2 \delta_{i,j}^{\alpha_i}))+\frac {\alpha_i\varrho(\xi_j)} 2 H^g(\cdot,\xi_j)+
		\begin{cases}
			\mathcal{O}(\delta_{i,j}^{2}|\log \delta_{i,j}|) & i=1\\
			\mathcal{O}(\delta_{i,j}^{2}) & i\geq 2 
		\end{cases}, $$ 
		as $\delta_{i,j}\rightarrow 0.$
		For any $x\in \Sigma\setminus\{ \xi_j\}$, 
		$$PU^i_j= \frac {\alpha_i\varrho(\xi_j)} 2 G^g(\cdot,\xi_j)+ 	\begin{cases}
			\mathcal{O}(\delta_{i,j}^{2}|\log \delta_{i,j}|) & i=1\\
			\mathcal{O}(\delta_{i,j}^{2}) & i\geq 2 
		\end{cases},$$
		as $\delta_{i,j}\rightarrow 0.$
	\end{lemma}
	\begin{proof}
		Let $\eta_{ij}= PU^i_j-\chi_j\cdot( U^i_j-\log( 2\alpha_i^2 \delta_{i,j}^{\alpha_i})) -\frac {\alpha_i\varrho(\xi_j)} 2 H^g(\cdot,\xi_j) $.
		If $\xi_j\in \intsigma$,  $\partial _{\nu_g}\eta_{ij} \equiv 0$ on $\partial \Sigma$.  
		We observe that for any $x\in \partial\Sigma\cap U(\xi_j)$, 
		$\partial_{\nu_g} |y_{\xi_j}(x)|^2=- e^{-\frac 1 2 \hat{\varphi}_{\xi_j}(y)} \left. \frac{\partial}{\partial y_2} |y|^2\right|_{y=y_{\xi_j}(x)}= 0.$
		If $\xi_j\in \partial \Sigma$,  
		for any $x\in \partial \Sigma$, we have  as $\delta_{i,j}\rightarrow 0$
		\begin{align*}
			\partial_{\nu_g}\eta_{ij}(x)
			&= 2\partial_{\nu_g}\left( \chi_j \log\left( 1+\frac{\delta_{i,j}^{\alpha_i}}{|y_{\xi_j}(x)|^{\alpha_i}}\right)\right)\\
			&=2(\partial_{\nu_g}\chi_{j}) \frac{\delta_{i,j}^{\alpha_i}}{|y_{\xi_j}(x)|^{\alpha_i}} -2\chi_{j} \partial_{\nu_g}\log\left( 1+\frac{\delta_{i,j}^{\alpha_i}}{|y_{\xi_j}(x)|^{\alpha_i}}\right)+\mathcal{O}(\delta_{i,j}^{2\alpha_i})\\
			&=\mathcal{O}(\delta_{i,j}^{\alpha_i}). 
		\end{align*}
		Thus, for any $i=1, \dots, N ,j=1,\ldots,m$,	$
		\partial_{\nu_g} \eta_{ij} =
		\mathcal{O}(\delta_{i,j}^{\alpha_i}) \text{ as } \delta_{i,j}\rightarrow 0.$
		\begin{align*}
			\int_{\Sigma} \eta_{ij} dv_g &=  \int_{\Sigma} 2\chi_j  \log\left( 1+\frac{\delta_{i,j}^{\alpha_i}}{|y_{\xi_j}|^{\alpha_i}}\right) dv_g= 2\int_{B^{\xi_j}_{r_0}} e^{\hat{\varphi}_{\xi_j}(y)}  \log\left( 1+ \frac{\delta_{i,j}^{\alpha_i} }{ |y|^{\alpha_i}}\right)  dy \\
			&\quad+2 \int_{B^{\xi_j}_{2r_0}\setminus \B_{r_0}} \chi\Big(\frac{|y|}{r_0}\Big) e^{\hat{\varphi}_{\xi_j}(y)}
			\left( \frac{\delta_{i,j}^{\alpha_i} }{ |y|^{\alpha_i}}+\mathcal{O}(\delta_{i,j}^{2\alpha_i})\right) dy\\
			&= 2\delta_{i,j}^2  \int_{\frac 1 {\delta_{i,j}}B^{\xi_j}_{r_0}} \log \left( 1+ \frac 1 {|y|^{\alpha_i}}\right) e^{\hat{\varphi}(\delta_{i,j} y)} dy +\mathcal{ O}(\delta_{i,j}^{\alpha_i})\\
			&=2\delta_{i,j}^2 \int_{\frac 1{\delta_{i,j}}B_{r_0}^{\xi_j} } (1+\mathcal{O}(\delta_{i,j}|y|)) \log \left( 1+\frac 1 {|y|^{\alpha_i}}\right) dy+\mathcal{O}(\delta_{i,j}^{\alpha_i}) \\
			&=\begin{cases}
				\cO(\delta_{i,j}^2|\log \delta_{i,j}|),  & i=1\\
				\cO(\delta_{i,j}^2),  & i\geq  2
			\end{cases},
		\end{align*}
		where we applied the fact that 
		\begin{align*}
			0&\leq\int_{|y|<\frac {r_0}{\delta_{i,j}}}     \log \left( 1+\frac 1 {|y|^{\alpha_i}}\right) dy 
			=2\pi \int_0^{r_0/\delta_{i,j}} \log   \left( 1+\frac 1 {r^{\alpha_i}}\right)r  dr\\
			&\leq \pi \int_0^{r^2_0/(\tau\rho)^2} \log   \left( 1+\frac 1 {t}\right) dt \leq 2\pi\int_1^{r_0/\delta_{i,j}} r^{1-\alpha_i} dr +\mathcal{O}(1)\\
			&= \begin{cases}
				\mathcal{O}(|\log \delta_{i,j}|) & i=1\\
				\cO(1) & i\geq 2
			\end{cases}.
		\end{align*}
		For any $x\in U_{2r_0}(\xi)$, $-\Delta_g U^i_j  =e^{-\varphi_{j}} e^{U^i_j}$.  It follows that 
		\begin{align*}
			-\Delta_{g} \eta_{ij}&=
			2(\Delta_g\chi_j )\log\frac{|y_{\xi_j}|^{\alpha_i}}{\delta_{i,j}^{\alpha_i}+|y_{\xi_j}|^{\alpha_i}}+4\Big\la \nabla\chi_j,\nabla\log\frac{|y_{\xi_j}|^{\alpha_i}}{\delta_{i,j}^{\alpha_i}+|y_{\xi_j}|^{\alpha_i}}\Big\ra_g\\
			&\quad +\frac 1 {|\Sigma|_g} \Big( \frac 1 2 \alpha_i\varrho(\xi_j)- \int_{\Sigma} \chi_j e^{-\varphi_j} |y_{\xi_j}|^{\alpha_i-2} e^{U^i_j} dv_g \Big). 
		\end{align*}
		We observe that $\Delta_{g} \chi_j\equiv 0$ and $\nabla \chi_j\equiv 0$ in $U_{2r_0}(\xi_j)\setminus U_{r_0}(\xi_j)$. 
		For any $x\in U_{2r_0}(\xi_j)\setminus U_{r_0}(\xi_j)$, we have 
		$ -2\log\Big( 1+ \frac{\delta_{i,j}^{\alpha_i}}{|y_{\xi_j}(x)|^{\alpha_i}}\Big)
		= -2\delta_{i,j}^{\alpha_i}  |y_{\xi}(x)|^{-{\alpha_i}} +\mathcal{O}(\delta_{i,j}^{2\alpha_i}) $
		and 
		\[-2\nabla \log\Big( 1+ \frac{\delta_{i,j}^{\alpha_i}}{|y_{\xi_j}(x)|^{\alpha_i}}\Big)
		= -2\delta_{i,j}^{\alpha_i} \nabla |y_{\xi_j}(x)|^{-\alpha_i} +\mathcal{O}(\delta_{i,j}^{2\alpha_i}), \] 
		as $\delta_{i,j}\rightarrow 0$.
		Moreover, a straightforward calculation of the integral implies that 
		\begin{align*}
			\int_{\Sigma} \chi_j e^{-\varphi_j}|y_{\xi_j}|^{\alpha_i-2} e^{U^i_j} dv_g 
			&	=\int_{B_{2r_0}^{\xi}} 2\alpha_i^2 \chi\Big(\frac{|y|}{r_0}\Big) \frac{\delta_{i,j}^{\alpha_i}|y|^{\alpha_i-2}}{(\delta_{i,j}^{\alpha_i} +|y|^{\alpha_i})^2} dy \\
			&=  \int_{B_{r_0}^{\xi}} 2\alpha_i^2 \frac{\delta_{i,j}^{\alpha_i}|y|^{\alpha_i-2}}{(\delta_{i,j}^{\alpha_i} +|y|^{\alpha_i})^2} dy +\mathcal{ O}(\delta_{i,j}^{\alpha_i} )\\
			&=\frac {\alpha_i\varrho(\xi_j)}2+{\mathcal{O}}(\delta_{i,j}^{\alpha_i}),
		\end{align*}
		where we applied the fact that $\int _{|y|<r}2\alpha_i^2 \frac{\delta_{i,j}^{\alpha_i}|y|^{\alpha_i-2}}{(\delta_{i,j}^{\alpha_i} +|y|^{\alpha_i})^2}   dy = 4\pi \alpha_i\Big(1- \frac{\delta_{i,j}^{\alpha_i}}{\delta_{i,j}^{\alpha_i}+r^{\alpha_i}}\Big)$ for any $r\geq 0.$
		Hence,  as $\delta_{i,j}\rightarrow 0$,
		$ -\Delta_g\eta_{ij}= \mathcal{O}(\delta_{i,j}^{\alpha_i}). $
		By the regularity theory for elliptic equations (see \cite{Agmon1959,Wehrheim2004}, for instance) and Sobolev embedding, we derive that as $\delta_{i,j}\rightarrow 0$,
		$
		\eta_{ij}=
		\begin{cases}
			\mathcal{O}(\delta_{i,j}^2|\log \delta_{i,j}|) & i=1\\
			\cO(\delta_{i,j}^2) & i\geq 2
		\end{cases} 
		$
		in $C(\Sigma)$. 
	\end{proof} 
	Using the same approach of Lemma~\ref{lem:extension_PU_as}, we can deduce the asymptotic expansion of $PZ_{ij}$ as follows:
	\begin{lemma}\label{lem:extension_PZ_as}
		For any $i=1, \dots, N $, $j=1,\ldots,m$, as $\delta_{i,j}\rightarrow 0$
		\begin{equation}
			\label{eq:extension_PZ}
			PZ_{ij}= Z_{ij}+ 1 + \mathcal{O}(\delta_{i,j}^{\alpha_i}|\log \delta_{i,j}|)= \frac{2\delta_{i,j}^{\alpha_i}}{\delta_{i,j}^{\alpha_i}+|y_{\xi_j}|^{\alpha_i}}+ \begin{cases}
				\cO(\delta_{i,j}^2|\log \delta_{i,j}|),  & i=1\\
				\cO(\delta_{i,j}^2),  & i\geq  2
			\end{cases}.
		\end{equation}
	\end{lemma}
	\begin{lemma}
		\label{lem:theta_small_as} 
		Let $\Theta_{ij}$ be defined by~\eqref{eq:def_theta_as}. For $i=1, \dots, N ,j=1,\ldots,m$, it holds 
		
		$
		|\Theta_{ij}(y)|=\mathcal{O}\left( \delta_{i,j}|y|+\varepsilon^{\frac 1 {2i}}\right), y\in\frac 1 {\delta_{i,j}}\cA_{ij},
		$
		and particularly,  $\sup_{\frac 1 {\delta_{i,j}}\cA_{ij}} |\Theta_{ij}(y)|=\mathcal{O}(1)$ as $\varepsilon\rightarrow 0.$
	\end{lemma}
	\begin{proof}
		Using Lemma~\ref{lem:extension_PU_as}  and the  Taylor expansion $f(\delta_{i,j}y)=1 +\cO(\delta_{i,j}|y|)$ for any $f\in C^1(\Omega_{ij})$, we deduce  that 
		\begin{align}\label{eq:A3-1}
			\Theta_{ij}(y)&= -\log(2\alpha_i^2 \delta_{i,j}^{\alpha_i}) +\frac{\alpha_i}{2}\Big( \varrho(\xi_j)R^g(\xi_j)+\sum_{\substack{j'=1\\ j'\neq j}}^m\varrho(\xi_{j'}) G^g(\xi_{j},\xi_{j'})\Big)\\
			&\quad+ \sum_{i'<i } \frac{a_{ii'}}2
			\Big( -2\alpha_{i'}\log(\delta_{i,j}|y|) + \frac{\alpha_{i'}}{2}
			\Big(\varrho(\xi_j)R^g(\xi_j)+\sum_{\substack{j'=1\\ j'\neq j}}^m\varrho(\xi_{j'}) G^g(\xi_{j},\xi_{j'})\Big)\nonumber\\
			&\quad+ \mathcal{O}\Big(\frac {\delta_{i',j}^{\alpha_{i'}}} {\delta_{i,j}^{\alpha_{i'}}} |y|^{-\alpha_{i'}}\Big) \Big)+\sum_{i'>i } \frac{a_{ii'}}2\Big( -2\alpha_{i'}\log\delta_{i',j}+ \frac{\alpha_{i'}}{2}\Big(\varrho(\xi_j)R^g(\xi_j)\nonumber\\
			&\quad +\sum_{\substack{j'=1\\ j'\neq j}}^m \varrho(\xi_{j'})G^g(\xi_{j'},\xi_j)\Big)+  \mathcal{O}\Big(\frac {\delta_{i,j}^{\alpha_{i'}}} {\delta_{i',j}^{\alpha_{i'}}}  |y|^{\alpha_{i'}} \Big) \Big) + \log V_i(\xi_j)+\log(a_{ii}\varepsilon)\nonumber\\
			&\quad -(\alpha_i-2)\log(\delta_{i,j} |y|) +\mathcal{O}\Big(
			\sum_{l=1}^N\sum_{j'=1}^m \varepsilon_{l,j'}+\delta_{i,j}|y|\Big)\nonumber\\
			&=-2\log\alpha_i -( \alpha_i \log d_{i,j} +\sum_{i'>i} a_{ii'} \alpha_{i'}\log (d_{i',j}) ) +\frac 1  2\Big(\alpha_i +\sum_{\sumi}^N \frac{a_{ii'}}{2} \alpha_{i'}\Big) \nonumber\\
			&\quad\times \Big(\varrho(\xi_j)R^g(\xi_j)+\sum_{\sumj}^m\varrho(\xi_{j'}) G^g(\xi_{j'},\xi_j)\Big)+ \log V(\xi_j)	-\Big(\alpha_i \log \delta_i \nonumber\\
			&\quad
			+\sum_{i'>i} a_{ii'} \alpha_{i'}\log \delta_{i'} -\log\varepsilon\Big)+\Big(- (\alpha_i-2)-\sum_{i'<i} a_{ii'} \alpha_i'\Big)\log( \delta_{i,j}|y|)\nonumber\\
			&\quad+ \sum_{i'<i}\mathcal{O}\Big(\frac {\delta_{i',j}^{\alpha_{i'}}} {\delta_{i,j}^{\alpha_{i'}}} |y|^{-\alpha_{i'}}\Big)+\sum_{i'>i} \cO \Big(\frac {\delta_{i,j}^{\alpha_{i'}}} {\delta_{i',j}^{\alpha_{i'}}} |y|^{\alpha_{i'}} \Big)+ \mathcal{O}\Big(\sum_{l=1}^N \sum_{j'=1}^m \varepsilon_{l,j'}+\delta_{i,j}|y|\Big),\nonumber
		\end{align}
		where $\varepsilon_{l,j'}$ denotes $\cO(\delta_{l,j'}^2 |\log \delta_{l,j'}|)$ for $l=1$ and  $\cO(\delta_{l,j'}^2)$ for $l\geq 2.$
		Recall $\alpha_i,\delta_{i,j}$ and $d_{i,j}$ are defined by~\eqref{def:alpha_i},~\eqref{def:delta_i} and~\eqref{eq:def_d_ij_as}, respectively. 
		It follows immediately that the first three terms in \eqref{eq:A3-1} vanish. 
		As a result, we  have that
		\begin{equation}
			\label{eq:theta_es1} \Theta_{ij}(y)=\sum_{i'<i}\mathcal{O}\Big(\frac {\delta_{i',j}^{\alpha_{i'}}} {\delta_{i,j}^{\alpha_{i'}}} |y|^{-\alpha_{i'}}\Big)+\sum_{i'>i} \cO \Big(\frac {\delta_{i,j}^{\alpha_{i'}}} {\delta_{i',j}^{\alpha_{i'}}} |y|^{\alpha_{i'}} \Big)+ \mathcal{O}\Big(\sum_{l=1}^N \sum_{j'=1}^m \varepsilon_{l,j'}+\delta_{i,j}|y|\Big).
		\end{equation} 
		The estimates	\eqref{def:alpha_i} and \eqref{def:delta_i} imply that 
		$ \mathcal{O}(\delta_{i,j}^{\alpha_i}|\log \delta_{i,j}|)=\mathcal{O}(\varepsilon\log |\varepsilon|), $
		for $i=1, \dots, N $ and $j=1,\ldots,m.$
		For $y\in \frac 1 {\delta_{i,j}}\cA_{ij}$, we have 
		$$ \sqrt{\delta_{i-1,j}/\delta_{i,j}}\leq |y|< \sqrt{\delta_{i+1,j}/\delta_{i,j}}.$$
		If $i'<i$, 	\begin{equation*}
			\begin{split}
				&\mathcal{O}\Big(\frac {\delta_{i',j}^{\alpha_{i'}}} {\delta_{i,j}^{\alpha_{i'}}|y|^{\alpha_{i'}}}\Big)=\mathcal{O}\Big( \Big( \frac{\delta_{i',j}^2}{\delta_{i,j}\delta_{i-1,j}}\Big)^{\alpha_{i'}/2}\Big)= \mathcal{O}\Big( \left( \frac{\delta_{i-1,j}}{\delta_{i,j}}\right)^{\alpha_{i'}/2}\Big)\stackrel{\eqref{eq:diff_deltai}}{=}\cO(\varepsilon^{\frac 1 {2i}});
			\end{split}
		\end{equation*}
		if $i'>i$, 
		\begin{equation*}
			\begin{split}
				&	\sum_{i'>i} \mathcal{O} \Big(\frac {\delta_{i,j}^{\alpha_{i'}}|y|^{\alpha_{i'}}} {\delta_{i',j}^{\alpha_{i'}}} \Big)=\mathcal{O}\Big( \Big( \frac{\delta_{i,j}\delta_{i+1,j}} {\delta_{i',j}^2}\Big)^{\alpha_{i'}/2}\Big)= \mathcal{O}\Big( \Big( \frac{\delta_{i,j}}{\delta_{i+1,j}}\Big)^{\alpha_{i'}/2}\Big)\stackrel{\eqref{eq:diff_deltai}}{=}\cO(\varepsilon^{\frac 1 {2i}}).
			\end{split}
		\end{equation*}
		Moreover, $\mathcal{O}(\delta_{i,j} |y|)=\mathcal{O}(1).$ Lemma~\ref{lem:theta_small_as} is complete.
	\end{proof}
	\begin{lemma}
		\label{lem:diif_e^W_sum_e^u_as} 
		For  $i=1, \dots, N ,$ there exists $p_0>1$ such that for any $p\in (1,p_0)$
		$$\Big\|2 \varepsilon V_ie^{W_{i,\varepsilon}}-\sum_{j=1}^m\chi_j e^{-\varphi_j}|y_{\xi_j}|^{\alpha_i-2} e^{U^i_j} \Big\|_{p}=\mathcal{O}\left(\varepsilon^{\frac{2-p}{4N}}\right),$$
		as $\varepsilon \rightarrow 0.$
	\end{lemma}
	\begin{proof}
		Applying Lemma~\ref{lem:extension_PU_as}, we have for any $x\in\Sigma\setminus \bigcup_{j=1}^m U_{r_0}(\xi_j)$,
		\begin{align*}
			W_{i,\varepsilon}(x)
			&= \sum_{j=1}^m PU^i_j +\sum_{\substack{i'=1 \\ i' \neq i}}^N \frac{a_{ii'}}2\sum_{j=1}^m PU^{i'}_j\\
			&=\sum_{j=1}^m  \frac{\varrho(\xi_j)}  2\Big( \alpha_i + \sum_{\substack{i'=1 \\ i' \neq i}}^N  \frac{a_{ii'}}2 \alpha_{i'} \Big)G^g(x,\xi_j)  +\mathcal{O}\Big(\sum_{i=1}^N \delta_{i,j}^{2}|\log\delta_{i,j}| \Big)\\
			&=\mathcal{O}(1).
		\end{align*}
		By straightforward calculation, we deduce that 
		\begin{align*}
			&\quad\int_{\Sigma} \Big|  2 \varepsilon V_ie^{W_{i,\varepsilon}}-\sum_{j=1}^m\chi_j e^{-\varphi_j}|y_{\xi_j}|^{\alpha_i-2} e^{U^i_j}\Big|^p  dv_g\\
			&=\sum_{j=1}^m \int_{U_{r_0}(\xi_j)} \Big| 2 \varepsilon V_ie^{W_{i,\varepsilon}}- e^{-\varphi_j}|y_{\xi_j}|^{\alpha_i-2} e^{U^i_j}  \Big|^p  d v_g + \mathcal{O}\left(\varepsilon^p+\delta_{i,j}^{\alpha_ip}\right)\\
			&=\sum_{j=1}^m \int_{\frac 1 {\delta_{i,j}}B^{\xi_j}_{r_0}(\xi_j)} \Big| 2\varepsilon V_i\circ y_{\xi_j}^{-1}(\delta_{i,j} y) e^{\varphi_{\xi_j}\circ y_{\xi_j}^{-1}(\delta_{i,j} y)}  e^{W_{i,\varepsilon}\circ y_{\xi_j}^{-1}(\delta_{i,j} y)}\\
			&\quad - |\delta_{i,j}y|^{\alpha_i-2} e^{U^i_j\circ y_{\xi_j}^{-1}(\delta_{i,j} y)}  \Big|^p  d v_g+ \mathcal{O}\left(\varepsilon^p+\delta_{i,j}^{\alpha_ip}\right) \\
			&=(2\alpha_i^2)^p\sum_{j=1}^m \int_{\Omega_{ij}} \frac{|\delta_{i,j}|^{2-2p}|y|^{(\alpha_i-2)p}}{(1+|y|^{\alpha_i})^{2p}}\left|e^{\Theta_{ij}(y)}-1 \right|^p  dy + \mathcal{O}\left(\varepsilon^p+\delta_{i,j}^{\alpha_ip}\right) . 
		\end{align*}
		We will compute the  integral 
		$\int_{\Omega_{ij}} \frac{|\delta_{i,j}|^{2-2p}|y|^{(\alpha_i-2)p}}{(1+|y|^{\alpha_i})^{2p}}\left|e^{\Theta_{ij}(y)}-1 \right|^p  dy $
		by dividing it into several regions: 
		$\Omega_{ij}\cap \frac 1 {\delta_{i,j}}\cA_{1j}, \dots,  \Omega_{ij}\cap \frac 1 {\delta_{i,j}}\cA_{Nj}$. 
		Lemma~\ref{lem:theta_small_as} yields that 
		\begin{align*}
			&\quad \int_{\Omega_{ij}\cap \frac 1 {\delta_{i,j}} \cA_{ij}} \frac{|\delta_{i,j}|^{2-2p}|y|^{(\alpha_i-2)p}}{(1+|y|^{\alpha_i})^{2p}}\left|e^{\Theta_{ij}(y)}-1 \right|^p  dy \\
			&= \mathcal{O}\Big(\int_{\Omega_{ij}\cap \frac 1 {\delta_{i,j}} \cA_{ij}} \frac{|\delta_{i,j}|^{2-2p}|y|^{(\alpha_i-2)p}}{(1+|y|^{\alpha_i})^{2p}}\left|\Theta_{ij}(y) \right|^p  dy \Big)\\
			&	= \mathcal{O}\Big(\int_{\Omega_{ij}\cap \frac 1 {\delta_{i,j}} \cA_{ij}} |\delta_{i,j}|^{2-2p}\frac{|y|^{(\alpha_i-2)p}}{(1+|y|^{\alpha_i})^{2p}}\left|\delta_{i,j}|y|+\varepsilon^{\frac 1{2i} } \right|^p  dy \Big)\\
			&=\mathcal{O}\Big( \delta_{i,j}^{2-p}+\delta_{i,j}^{2-2p}
			\varepsilon^{\frac 1 {2pi}} \Big)
			\stackrel{\eqref{def:delta_i}}{=}\mathcal{O}(\varepsilon^{(2-p)\frac{1}{4i}} +\varepsilon^{\frac 1 {2ip}- \frac{(p-1)N}{2}})=\mathcal{O}(\varepsilon^{\frac{2-p}{4i}}),
		\end{align*}
		for $p>1$ sufficiently close to $1$.
		For $i'\neq i,$ by \eqref{def:delta_i}
		\begin{align}\label{eq:2}
			&\quad\int_{\Omega_{ij}\cap \frac 1 {\delta_{i,j}} \cA_{i'j}} \frac{|\delta_{i,j}|^{2-2p}|y|^{(\alpha_i-2)p}}{(1+|y|^{\alpha_i})^{2p}}\left|e^{\Theta_{ij}(y)}-1 \right|^p  dy \\
			&\leq 
			2^p\int_{\Omega_{ij}\cap \frac 1 {\delta_{i,j}} \cA_{i'j}} \frac{|\delta_{i,j}|^{2-2p}|y|^{(\alpha_i-2)p}}{(1+|y|^{\alpha_i})^{2p}}  dy\nonumber
			\\
			&\quad	+2^p \int_{\Omega_{ij}\cap \frac 1 {\delta_{i,j}} \cA_{i'j}} \frac{|\delta_{i,j}|^{2-2p}|y|^{(\alpha_i-2)p}}{(1+|y|^{\alpha_i})^{2p}}\left|e^{\Theta_{ij}(y)}\right|^p  dy.\nonumber
		\end{align}
		For the first term on the $R.H.S.$ of~\eqref{eq:2}, applying~\eqref{def:delta_i} and~\eqref{eq:diff_deltai}, we obtain
		\begin{equation*}
			\begin{split}
				&\int_{\Omega_{ij}\cap \frac 1 {\delta_{i,j}} \cA_{i'j}} \frac{|\delta_{i,j}|^{2-2p}|y|^{(\alpha_i-2)p}}{(1+|y|^{\alpha_i})^{2p}}  dy\\
				&=\mathcal{O}\Big( \int_{\frac{\sqrt{\delta_{i'-1,j}\delta_{i',j}}}{\delta_{ij}}\leq  |y|\leq \frac{\sqrt{\delta_{i',j}\delta_{i'+1,j}}}{\delta_{ij}}} \frac{|\delta_{i,j}|^{2-2p}|y|^{(\alpha_i-2)p}}{(1+|y|^{\alpha_i})^{2p}}  dy\Big)\\
				&=\left\{\begin{array}{c}
					O\Big(\delta_{i,j}^{2-2 p}\left(\frac{\sqrt{\delta_{i',j} \delta_{i'+1,j}}}{\delta_{i,j}}\right)^{\left(\alpha_{i}-2\right) p+2}\Big)=O\Big(\delta_{i,j}^{2-2 p}\left(\frac{\delta_{i',j}}{\delta_{i'+1,j}}\right)^{\frac{\left(\alpha_i-2\right) p+2}{2}}\Big) 
					\text { if } i>i' \\
					O\Big(\delta_{i,j}^{2-2 p}\left(\frac{\delta_{i,j}}{\sqrt{\delta_{i'-1,j} \delta_{i',j}}}\right)^{\left(\alpha_j+2\right) p-2}\Big)=O\Big(\delta_{i,j}^{2-2 p}\Big(\frac{\delta_{i'-1,j}}{\delta_{i',j}}\Big)^{\frac{(\alpha_i+2) p-2}{2}}\Big) 
					\text { if } i<i' 
				\end{array}\right.\\
				&= \left\{\begin{array}{c}
					O\left(
					\varepsilon^{-N(p-1) + \frac 1 {2N} }\right) 
					\text { if } i>i' \\
					O\left(\varepsilon^{-N(p-1) +\frac{2p-1}{2N}}\right) 
					\text { if } i<i' 
				\end{array}\right.=\mathcal{O}(\varepsilon^{\frac{2-p}{4N}}),
			\end{split}
		\end{equation*}
		for $p>1$ sufficiently close to $1$. 
		To estimate the second term of~\eqref{eq:2}, by  Lemma \ref{lem:extension_PU_as} we derive that 
		\begin{align*}
			&\quad	\int_{\Omega_{ij}\cap \frac 1 {\delta_{i,j}} \cA_{i'j}} \frac{|\delta_{i,j}|^{2-2p}|y|^{(\alpha_i-2)p}}{(1+|y|^{\alpha_i})^{2p}}\left|e^{\Theta_{ij}(y)}\right|^p  dy\\
			&=  \int_{\delta_{i,j}\Omega_{ij}\cap  \cA_{i'j}} |2 \varepsilon V_ie^{W_{i,\varepsilon}}( y_{\xi_j}^{-1} (y)) |^p dy=\mathcal{O}\left( \cI_j^{(i,i')}\right),
		\end{align*}
		where $\cI_j^{(i,i')}:= \varepsilon^p  \int_{\delta_{i,j}\Omega_{ij} \cap  \cA_{i'j}}\Big(\frac{\Pi_{l\neq i} (\delta_{l,j}^{\alpha_{l}}+ |y|^{\alpha_{l}})^{-a_{il}} }{(\delta_{i,j}^{\alpha_i}+|y|^{\alpha_i})^2}\Big)^p dy.$ 
		
		We compute $\cI_j^{(i,i')}$ by considering three distinct cases for the index $i$:
		\begin{itemize}
			\item 	[(1)] $i = 1$;
			\item [(2)] $i = N$;
			\item [(3)] $2 \leq i \leq N - 1$.
		\end{itemize}

		For $i=1,$ we have $i'\geq 2$ and $a_{il}=0$ for all $l>2$. It follows that 
		\begin{align*}
			\cI_j^{(i,i')}
			&= \varepsilon^p \int_{\delta_{1,j} \Omega_{1j}\cap\cA_{i'j}} \Big(\frac{(\delta_{2,j}^{\alpha_2}+ |y|^{\alpha_2})^{-a_{12}}}{(\delta_{1,j}^{\alpha_1}+ |y|^{\alpha_1})^2 }\Big)^p dy\\
			&=\cO \Big( \varepsilon^p
			\delta_{1,j}^{2-2\alpha_1 p} \delta_{2,j}^{(-a_{12})\alpha_2p} \int_{\Big\{|y|>\frac{\sqrt{\delta_{i',j}\delta_{i'-1,j}}}{\delta_{1,j}}\Big\}}  \frac{(1+\delta_{1,j}^{\alpha_2}\delta_{2,j}^{-\alpha_2} |y|^{\alpha_2})^{-a_{12} p}}{(1+ |y|^{\alpha_1})^{2p} } dy  \Big)\\
			&=\cO\Big( \varepsilon^p (\delta_1\delta_2) \frac{\delta_2^{(-a_{12})\alpha_2p}}{(\delta_1\delta_2)^{\alpha_1 p}} + \varepsilon^p (\delta_1\delta_2)^{ 1-\alpha_1 p + \frac{(-a_{12})\alpha_2 p}{2}}\Big)\\
			&=\cO( \varepsilon^{(\frac{N}{2}( \frac 1 p-1)+\frac  12 )p})=\cO(\varepsilon^{\frac p 4}),
		\end{align*}
		for $p>1$ sufficiently small, 
		where we applied \eqref{def:alpha_i} and \eqref{def:delta_i}. 
		
		For $i=N$, we have $i'\leq N-1$ and $a_{il}=0$ for all $l<N-1$. It follows that 
		\begin{equation*}
			\begin{split}
				\cI_j^{(i,i')}&= \varepsilon^p \int_{\delta_{N,j} \Omega_{Nj}\cap\cA_{i'j}} \Big(\frac{(\delta_{N-1,j}^{\alpha_{N-1}}+ |y|^{\alpha_{N-1}})^{-a_{NN-1}}}{(\delta_{N,j}^{\alpha_N}+ |y|^{\alpha_N})^2 }\Big)^p dy\\
				&=\cO \Big( \varepsilon^p
				\delta_{N,j}^{2-2\alpha_N p}  \int_{\Big\{|y|<\frac{\sqrt{\delta_{i',j}\delta_{i'+1,j}}}{\delta_{N,j}}\Big\}}  \Big(\frac{\delta_{N-1,j}^
					{(-a_{NN-1})\alpha_{N-1}}} {(1+ |y|^{\alpha_N})^2 } \Big)^p dy\\
				&\quad+\varepsilon^p 	\delta_{N,j}^{2-2\alpha_N p-a_{NN-1}\alpha_{N-1}p}
				\int_{\Big\{|y|<\frac{\sqrt{\delta_{i',j}\delta_{i'+1,j}}}{\delta_{N,j}}\Big\}}   |y|^{(-a_{NN-1})\alpha_{N-1}p} dy  \Big)\\
				&=\cO\Big( \varepsilon^p (\delta_N\delta_{N-1}) \frac{\delta_{N-1}^{(-a_{NN-1})\alpha_{N-1}p}}{\delta_N^{2\alpha_N P}} +    \Big(\frac{\delta_{N-1}}{\delta_N}\Big)^{ (1-\frac 1 2 (-a_{NN-1}) \alpha_{N-1})p}\Big)\\
				&=\cO(\varepsilon^p+ \varepsilon^{\frac 1 {2(N-1)}\frac {\alpha_N} 2 p} )=\cO(\varepsilon^{\frac p 2}),
			\end{split}
		\end{equation*}
		where we applied \eqref{def:alpha_i} and \eqref{def:delta_i}. 
		
		For the case $2\leq i\leq N-1$, we have $N\geq 3$ and $|a_{il}|=0$ for all $|i-l|\geq 2.$ 	Since $\delta_{l,j}$ is monotonically increasing with respect to the index  $i$ for $\varepsilon$ sufficiently small, 
		$\cI_j^{(i, i')}$ can be reduced as follows: for $i< i'$ 
		\begin{equation*}
			\begin{split}
				\cI^{(i, i')}_j&= 
				\cO \Big( \varepsilon^p 
				\Pi_{\substack{ l\geq i'+1\\ l\neq i}  } \delta_{l,j}^{p(-a_{il} \alpha_l)}  \int_{\delta_{i,j}
					\Omega_{ij} \cap  \cA_{i'j}} \Pi_{\substack{ l'\leq i'-1\\ l'\neq i} } |y|^{  p(-a_{il'} {\alpha_{i'}} )} 
				\\
				&\quad \times 
				\Big(\frac{ (\delta_{i',j}^{\alpha_{i'}}+ |y|^{\alpha_{i'}})^{-a_{ii'}} } {(\delta_{i,j}^{\alpha_i}+|y|^{\alpha_i})^2}\Big)^p dy \Big)\\
				&=\cO\Big( \varepsilon^p 	\Pi_{\substack{  l\neq i}  } \delta_{l,j}^{p(-a_{il} \alpha_l)}  \delta_i^{2-2\alpha_ip} \int_{\frac{\cA_{i'j}}{\delta_{i,j}} } \Big( \frac{ |y|^{ (-a_{ii-1})\alpha_{i-1}+(-a_{ii+1})\alpha_{i+1}} }{(1+|y|^{\alpha_i})^2}\Big)^p dy\Big)\\
				&\quad+\cO(\varepsilon^p \Pi_{\substack{  l\neq i}  } \delta_{l,j}^{(-a_{il} \alpha_l)p}  \delta_i^{2-2\alpha_ip} ).
			\end{split}
		\end{equation*}
		Due to the constructions \eqref{def:alpha_i} and \eqref{def:delta_i}, we have  $\delta_{i-1}^{\alpha_{i-1}}=\varepsilon^{N+2-i}$, $\delta_{i}^{\alpha_{i}}=\varepsilon^{N+1-i}$
		and $\delta_{i+1}^{\alpha_{i+1}}=\cO(\varepsilon^{N-i})$, which implies that
		\begin{equation}\label{eq:est_i3}
			\Pi_{\substack{  l\neq i}  } \delta_{l,j}^{p(-a_{il} \alpha_l)}  \delta_i^{-2\alpha_ip}=\cO(1).
		\end{equation} 
		If $i>i'$,  we can immediately deduce that 
		\[ \cI_j^{(i,i')}=\cO\Big( \varepsilon^p 	 \delta_{i-1,j}^{(-a_{ii-1} \alpha_{i-1})p} \delta_{i+1,j}^{(-a_{ii+1} \alpha_{i+1})p}  \delta_i^{2-2\alpha_ip}  \Big)\stackrel{\eqref{eq:est_i3}}{=}\cO(\varepsilon^p).\]
		Next, we consider the case $i<i'$. It holds  
		\begin{equation*}
			\begin{split}
				\cI^{(i, i')}_j&= 
				\cO \Big( \varepsilon^p 
				\Pi_{\substack{ l\geq i'+1\\ l\neq i}  } \delta_{l,j}^{p(-a_{il} \alpha_l)}  \int_{\delta_{i,j}
					\Omega_{ij} \cap  \cA_{i'j}} \Pi_{\substack{ l'\leq i'-1\\ l'\neq i} } |y|^{  p(-a_{il'} {\alpha_{i'}} )} 
				\\
				&\quad\times 
				\Big(\frac{ (\delta_{i',j}^{\alpha_{i'}}+ |y|^{\alpha_{i'}})^{-a_{ii'}} } {(\delta_{i,j}^{\alpha_i}+|y|^{\alpha_i})^2}\Big)^p dy \Big)\\
				&=\cO\Big(\varepsilon^p \int_{\delta_{i,j}
					\Omega_{ij} \cap  \cA_{i'j}} \Pi_{\substack{ l'\leq i'-1\\ l'\neq i} } |y|^{  p(-a_{il'} {\alpha_{i'}} )} 	\Big(\frac{ (\delta_{i',j}^{\alpha_{i'}}+ |y|^{\alpha_{i'}})^{-a_{ii'}} } {(\delta_{i,j}^{\alpha_i}+|y|^{\alpha_i})^2}\Big)^p dy  \Big)\\
				&= \cO\Big( \varepsilon^p \int_{\delta_{i,j}
					\Omega_{ij} \cap  \cA_{i'j}}\frac{1} {|y|^{(2\alpha_i- (-a_{ii-1})\alpha_{i-1}-(-a_{ii+1})\alpha_{i+1})p}} dy\Big)\\
				&=\cO(\varepsilon^p),
			\end{split}
		\end{equation*}
		where we applied that $2 -(2\alpha_i- (-a_{ii-1})\alpha_{i-1}-(-a_{ii+1})\alpha_{i+1})p>0$. 
		
		Hence, we obtain that  for any $i\neq i'$
		\[ \int_{\Omega_{ij}\cap \frac 1 {\delta_{i,j}} \cA_{i'j}} \frac{|\delta_{i,j}|^{2-2p}|y|^{(\alpha_i-2)p}}{(1+|y|^{\alpha_i})^{2p}}\left|e^{\Theta_{ij}(y)}\right|^p  dy =\mathcal{O}(\varepsilon^{\frac{2-p}{4N}}). \]
	\end{proof}
	\begin{lemma}~\label{lem:non_linear_basic}
		For any $p\geq 1$ and $r>1$, there are positive constants $c_1, c_2$ such that for any $\varepsilon>0$, the following estimates hold for any $\phi_1, \phi_2\in \oH$ and any $i=1, \dots, N $:
		\begin{equation}~\label{diff1_as}
			\| \varepsilon V_i e^{W_{i,\varepsilon} }(e^{\phi_1} -1-\phi_1)\|_{p}\leq c_1 \varepsilon^{\frac 1 {4N}\frac{2-2pr}{pr}}   e^{c_2 \|\phi_1\|^2}  \|\phi_1\|^2,
		\end{equation}
		and 
		\begin{align}
			~\label{diff1_as2}
			&\quad\| \varepsilon V_i e^{W_{i,\varepsilon} }(e^{\phi_1}-e^{\phi_2} -(\phi_1-\phi_2))\|_{p}\\
			&\leq c_1 \varepsilon^{\frac 1 {4N}\frac{2-2pr}{pr}} e^{c_2\sum_{h=1}^2\|\phi_h\|}  (\sum_{h=1}^2\|\phi_h\|) \|\phi_1-\phi_2\|. \nonumber
		\end{align}
	\end{lemma}
	\begin{proof}
		By the mean value theorem,  for some $s\in (0,1)$
		\[ |e^{\phi_1}-e^{\phi_2} -\phi_1+\phi_2|\leq \left| e^{s\phi_1+(1-s)\phi_2} -1\right||\phi_1-\phi_2|\leq  e^{\sum_{h=1}^2 |\phi_h|} |\phi_1-\phi_2|\sum_{h=1}^2 |\phi_h|.\]
		The H\"{o}lder's inequality, Sobolev inequality, and Moser-Trudinger inequality yield that 
		\begin{align*}
			&\quad	\Big( \int_{\Sigma} V_i^pe^{p \sum_{j=1}^m\Big(PU^i_j+ \sum_{\sumi}^N \frac{a_{ii'}} 2PU^{i'}_j\Big)} | e^{\phi_1} -e^{\phi_2}-(\phi_1-\phi_2)| ^p dv_g\Big)^{1/p} \\
			&\leq C \sum_{h=1}^2
			\Big(  \int_{\Sigma}
			V_i^p	e^{p \sum_{j=1}^m\Big(PU^i_j+ \sum_{\sumi}^N \frac{a_{ii'}} 2  PU^{i'}_j\Big) }(  e^{|\phi_1|+|\phi_2|}  |\phi_1-\phi_2||\phi_h| ) ^p dv_g \Big)^{1/p}  \\
			&\leq  C \sum_{h=1}^2 \Big(\int_{\Sigma} V_i^{pr}e^{pr\Big(PU^i_j+ \sum_{\sumi}^N \frac{a_{ii'}} 2PU^{i'}_j\Big)} dv_g\Big)^{\frac{1}{pr}} \Big( \int_{\Sigma}   e^{ps (|\phi_1|+|\phi_2|)} dv_g\Big)^{\frac{1}{ps}}\\
			&\quad \times \Big(\int_{\Sigma} |\phi_1-\phi_2|^{pt} |\phi_h|^{pt} dv_g\Big)^{\frac{1}{pt}}\\
			&\leq  C \sum_{h=1}^2 \Big(\int_{\Sigma} V_i^{pr}e^{pr\Big(PU^i_j+ \sum_{\sumi}^N \frac{a_{ii'}} 2PU^{i'}_j\Big)} dv_g\Big)^{\frac{1}{pr}} e^{\frac{ps}{8\pi}(\sum_{h=1}^{2}\|\phi_h\|^2)}   \|\phi_1-\phi_2\| \|\phi_h\|,
		\end{align*}
		where $ r,s,t \in (1, +\infty), {  \frac{1}{r}+\frac{1}{s}+\frac{1}{t}=1}$. 
		Applying Lemma~\ref{lem:diif_e^W_sum_e^u_as}, we deduce that 
		\begin{equation*}
			\begin{split}
				\|2\varepsilon V_i e^{W_{i,\varepsilon}}\|_{pr}
				&\leq \Big\|2 \varepsilon V_ie^{W_{i,\varepsilon}}-\sum_{j=1}^m\chi_j e^{-\varphi_j}|y_{\xi_j}|^{\alpha_i-2} e^{U^i_j} \Big\|_{pr}\\
				&\quad+ \sum_{j=1}^m\Big\|\chi_j e^{-\varphi_j}|y_{\xi_j}|^{\alpha_i-2} e^{U^i_j} \Big\|_{pr}\\
				&\leq \sum_{j=1}^m\Big\|\chi_j e^{-\varphi_j}|y_{\xi_j}|^{\alpha_i-2} e^{U^i_j} \Big\|_{pr}+\mathcal{O}(\varepsilon^{\frac 1 {4N}\frac{2-pr}{ pr}}).
			\end{split}
		\end{equation*}
		By Lemma~\ref{lem:extension_PU_as} and the definition in \eqref{def:delta_i}, it follows that
		\begin{align*}
			\int_{\Sigma} \chi^{pr}_j e^{-pr\varphi_j}|y_{\xi_j}|^{pr(\alpha_i-2)} e^{pr U^i_j}dv_g&=  \mathcal{O}\Big(\delta_{i,j}^{2-\alpha_i pr}\int_{\Omega_{ij}} \Big(\frac{ |y|^{\alpha_i-2}}{(1+|y|^{\alpha_i})^2} \Big)^{pr} dy \Big)+\mathcal{O}(\delta_{i,j}^{\alpha_i rp})\\
			&= \mathcal{O}(\delta_{i,j}^{2-2pr})= \mathcal{O}( \varepsilon^{\frac{2-2pr}{4N}}).
		\end{align*}
		By setting $\phi_2\equiv 0$, \eqref{diff1_as} follows immediately. 
	\end{proof} 
	
	%%%%%%%%%%%%%%%%%%%%%%%%%%%%%%%%%%%%%%%%%%%%%%%%%%%%%%%%%%%%

	%The acknowledgments section should not be numbered.
	%	\section*{Acknowledgments}
	%	We would like to thank 

	%%%%%%%%%%%%%%%%%%%%%%%%%%%%%%%%%%%%%%%%%%%%%%%%%%%%%%
	%          7. REFERENCES SECTION
	%%%%%%%%%%%%%%%%%%%%%%%%%%%%%%%%%%%%%%%%%%%%%%%%%%%%%%

	%\bibliographystyle{plain}
	%\bibliography{ref_Toda.bib}

	%%%%%%%%%%%%%Information of authors%%
	\medskip
	\noindent Zhengni Hu
	
	\noindent  School of Mathematical Sciences, Shanghai Jiao Tong University\\
	800 Dongchuan Road, Shanghai, 200240,  P. R. China \\ 
	Email: \textsf{	\href{mailto: zhengni_hu2021@outlook.com}{ zhengni\_hu2021@outlook.com}}
	
	\vskip10pt
	
	\noindent Miaomiao Zhu
	
	\noindent  School of Mathematical Sciences, Shanghai Jiao Tong University\\
	800 Dongchuan Road, Shanghai, 200240, P. R. China \\ 
	Email: \textsf{	\href{mailto: mizhu@sjtu.edu.cn}{ mizhu@sjtu.edu.cn}}
	
\end{document}